\documentclass[a4paper,10pt,reqno]{amsart}

\usepackage[UKenglish]{babel}
\usepackage{amsmath}
\usepackage{amssymb}
\usepackage{amsthm}
\usepackage{mathtools}
\usepackage{verbatim}
\usepackage{stackrel}
\usepackage[arrow, matrix, curve]{xy} %for diagrams of implications

\usepackage[shortlabels]{enumitem}

\usepackage{comment,cite}	
\usepackage{hyperref}
\hypersetup{
    colorlinks=true,
    linkcolor=blue,
    citecolor=red,
    filecolor=magenta,      
    urlcolor=cyan,
    bookmarks=true,
    linktocpage=true,
}

\usepackage[utf8]{inputenc}

\usepackage{amssymb}
\usepackage{amsmath}
\usepackage{amsthm}
\usepackage{bbm}
\usepackage{verbatim, stackrel}

\usepackage[shortlabels]{enumitem}
\usepackage{marginnote}
\usepackage{color}

\usepackage{stmaryrd}
\usepackage{tikz}
\usepackage{tikz-cd}
\usepackage{adjustbox}
\usepackage{float}

\usepackage{tabu,multirow}
\usepackage{makecell}
\usepackage{hhline}

\usepackage{todonotes}

\newcommand{\bbC}{\mathbb{C}}
\newcommand{\bbD}{\mathbb{D}}

\newcommand{\bbK}{\mathbb{K}}

\newcommand{\bbN}{\mathbb{N}}

\newcommand{\bbR}{\mathbb{R}}

\newcommand{\bbT}{\mathbb{T}}

\newcommand{\bbZ}{\mathbb{Z}}

\newcommand{\calL}{\mathcal{L}}

\newcommand{\tmatrix}[1]{\left(\begin{smallmatrix}#1\end{smallmatrix}\right)}

\DeclareMathOperator{\FBL}{FBL}
\DeclareMathOperator{\FBLC}{FBL_\bbC}

\DeclareMathOperator{\id}{id} % identity operator
\DeclareMathOperator{\one}{{\mathbbm{1}}} % constant function with value one
\DeclareMathOperator{\re}{Re} % real part
\newcommand{\argument}{\mathord{\,\cdot\,}} % argument dot for functions (with correct spacing)
\newcommand{\dx}{\;\mathrm{d}} % differential (for use at the end of integrals)
\newcommand{\norm}[1]{\left\lVert #1 \right\rVert} % norm
\newcommand{\modulus}[1]{\left\lvert #1 \right\rvert} % modulus
\newcommand{\iu}{{\mathrm i}} % imaginary unit

\newcommand{\spec}{\sigma} % spectrum

\newcommand{\appSpec}{\spec_{\operatorname{app}}} % approximate point spectrum
\newcommand{\pntSpec}{\spec_{\operatorname{pnt}}} % point spectrum
\newcommand{\spr}{r} % spectral radius
\newcommand{\pntSpr}{\spr_{\operatorname{pnt}}} 
 
\newcommand{\spri}{q}
\newcommand{\pntSpri}{\spri_{\operatorname{pnt}}} 
\newcommand{\appSpri}{\spri_{\operatorname{app}}} 
\theoremstyle{definition}
\newtheorem{definition}{Definition}[section]
\newtheorem{convention}[definition]{Convention}
\newtheorem{remark}[definition]{Remark}

\newtheorem*{remark*}{Remark}
\newtheorem*{remarks*}{Remarks}
\newtheorem{example}[definition]{Example}
\newtheorem{examples}[definition]{Examples}

\theoremstyle{plain}
\newtheorem{proposition}[definition]{Proposition}
\newtheorem{lemma}[definition]{Lemma}
\newtheorem{theorem}[definition]{Theorem}
\newtheorem{corollary}[definition]{Corollary}

\newtheorem*{duplicate*}{Theorem~\ref{thm:spec-char-general}}

\numberwithin{equation}{section} % enumerate formulas within sections

\begin{document}

\title{The spectrum of operator extensions to free Banach Lattices}
\author{Jochen Glück}
\address[J.\ Glück]{University of Wuppertal, School of Mathematics and Natural Sciences, Gaußstr.\ 20, 42119 Wuppertal, Germany}
\email{glueck@uni-wuppertal.de}

\author{Phillip Krokor}
\address[P.\ Krokor]{University of Wuppertal, School of Mathematics and Natural Sciences, Gaußstr.\ 20, 42119 Wuppertal, Germany}
\email{krokor@uni-wuppertal.de}
\subjclass[2010]{}
\date{\today}
\begin{abstract} 
    Every bounded linear operator $T$ on a complex Banach space $E$ is known to extend to a lattice homomorphism $\overline{T}$ that acts on the so-called complex free Banach lattice over $E$. 
    We show the following three results about the spectrum $\sigma(\overline{T})$ of $\overline{T}$: 
    (i)~$\sigma(\overline{T})$ always contains the spectrum $\sigma(T)$;
    this answers a recent question of de Hevia and Tradacete. 
    (ii)~It can happen that $\sigma(T)$ is a singleton while $\sigma(\overline{T})$ is the entire unit circle;
    This shows that $\spec(\overline{T})$ is not the closure of the cyclic hull of $\sigma(T)$ in general.
    (iii)~Finally, we give a full characterization of $\sigma(\overline{T})$ in terms of the spectral properties of $T$.

    Some of our arguments also give new results about the spectral properties of general lattice homomorphisms. 
    As a main tool we make extensive use of the Banach lattice functional calculus for continuous positively homogeneous functions.
\end{abstract}

\maketitle

\section{Introduction}

\subsection*{Spectral theory and cyclicity}

Inspired by the classical theorems of Perron and Frobenius an intricate spectral theory has been developed for positive operators on Banach lattices over many decades. 
The notion of a cyclic subset of $\bbC$ is particularly relevant in this theory: 
a set $S \subseteq \bbC$ is called \emph{cyclic} if for all $r \in (0,\infty)$ and $\theta \in \bbR$ the property $re^{\iu \theta}\in S$ implies that $r e^{\iu n \theta} \in S$ for all $n \in \bbZ$. 
Note that, in particular, a subset of the complex unit circle $\bbT$ is cyclic if and only if it is a union of multiplicative subgroups of $\bbT$. 
It is a classical result of Scheffold that the spectrum $\spec(T)$ of every lattice homomorphism $T \in \calL(E)$ on a complex Banach lattice $E$ is cyclic (see \cite[Theorem~2.3]{Scheffold1971} or, for a reference in English, \cite[Theorem~V.4.4 on p.\,325]{Schaefer1974}). 

If $T$ is only a positive operator rather than a lattice homomorphism, the so-called \emph{peripheral spectrum} $\{\lambda \in \spec(T) \mid \modulus{\lambda} = \spr(T)\}$, where $\spr(T)$ denote the spectral radius of $T$, is still known to be cyclic under various assumptions; see e.g. \cite[Theorem~V.4.9 and its Corollary on p.\,327--328]{Schaefer1974}, \cite[Satz~3.6]{Scheffold1971}, \cite[Section~2.2]{Zhang1991} and \cite[Theorem~7.1]{Glueck2016a}. 
It is an open problem whether the cyclicity of the peripheral spectrum holds for all positive operators on Banach lattices without additional assumptions; see \cite{Glueck2018} for a detailed discussion of this problem.

This paper is a contribution to the spectral theory on a special class of Banach lattices, the so-called \emph{free Banach lattices}.

\subsection*{Free Banach lattices over $\bbR$}

Free Banach lattices generated by a Banach space are a  concept developed by Avilés, Rodríguez and Tradacete in 2018 \cite{AvilesRodriguezTradacete2018}, motivated by free Banach lattices over a set \cite{PagterWickstreadAnthony2015} and by the much earlier concept of a free vector lattice \cite{Baker1968, Bleier1973}.  

As is common for free objects, free Banach lattices are defined by a universal property.
Let $E$ be a real Banach lattice. 
The free Banach lattice generated by $E$ is a Banach lattice $\FBL[E]$ together with an isometric linear embedding $\delta_E:E\hookrightarrow \FBL[E]$ such that the following universal property holds: 
for every Banach lattice $X$ and every bounded linear operator $T:E\to X$ there exists a unique Banach lattice homomorphism $\hat T: \FBL[E]\to X$ such that $\norm{T} = \norm{\hat T}$ and $T = \hat T \circ\delta_E$.
The $\FBL[E]$ can be shown to exist and to be unique up to isometric isomorphism. 
As a consequence, for every bounded linear operator $T: E \to E$ there exists a unique Banach lattice homomorphism $\overline{T}: \FBL[E] \to \FBL[E]$ that satisfies $\delta_E T = \overline{T}\delta_E$.

The theory of free Banach lattices has developed quickly and by now contains results about its structural representation \cite[Theorem 2.5 on p.\,2958]{AvilesRodriguezTradacete2018}, p-convexity \cite[Theorem 2.1 on p.\,17]{OikhbergTaylorTradaceteTroitsky2022}, its dual space \cite[Theorem 3.4 on p.\,8]{GarciaSanchezTradacete2024}, and a generalization to free Banach lattices over pre-ordered Banach spaces \cite[Theorem~6.9 on p.\,30]{DeJeuJiang2026} among other topics.

\subsection*{Free Banach lattices over $\bbC$}

As the complex scalar field is more suitable for spectral theory than the real one, it is natural to also develop a theory of free complex Banach lattices, as was recently done in \cite{deHeviaTradacete2023}.

Let $E$ be a complex Banach space. 
The free complex Banach lattice generated by $E$ is a complex Banach lattice $\FBL_\bbC[E]$ together with a $\bbC$-linear isometric embedding $\delta_E: E\hookrightarrow \FBL_\bbC[E]$ such that, similarly to the real case, the following universal property holds: 
for every complex Banach lattice $X_\bbC$ and every bounded $\bbC$-linear operator $T:E\to X_\bbC$ there is a unique Banach lattice homomorphism $\hat T: \FBL_\bbC[E]\to X_\bbC$ such that $T = \hat T \circ \delta_E$ and $\norm{\hat T} = \norm{T}$. 
As in the real case, it follows that for every bounded linear operator $T: E \to E$ there exists a unique Banach lattice homomorphism $\overline{T}: \FBL_\bbC[E] \to \FBL_\bbC[E]$ that satisfies $\delta_E T = \overline{T}\delta_E$, and this operator also satisfies $\norm{\overline{T}} = \norm{T}$; 
see the preliminaries at the end of the introduction for a presentation of this situation as a commutative diagram.

The Banach lattice $\FBL_\bbC[E]$ can be constructed as $\FBL_\bbC[E] \coloneqq \FBL[E_\bbR]\oplus i\FBL[E_\bbR]$ and $\delta_E\coloneqq \delta_{E_\bbR} - i\delta_{E_\bbR}$ where $E_\bbR$ is just $E$ interpreted as a real space ignoring the complex structure (see \cite[Theorem~3.3]{deHeviaTradacete2023}).

\subsection*{Contributions}

For a bounded linear operator $T$ on a complex Banach space $E$, we study the spectrum $\spec(\overline{T})$ of the lattice homomorphism $\overline{T}: \FBL_\bbC[E] \to \FBL_\bbC[E]$ and its relation to $\spec(T)$.
Questions of this type were first studied by de Hevia and Tradacete who asked whether the inclusion $\spec(T) \subseteq \spec(\overline{T})$ always holds and proved several partial results on this question \cite[Section~6]{deHeviaTradacete2023}. 

We settle this question in Section~\ref{sec:spectral-inclusion} where we show that the inclusion $\spec(T) \subseteq \spec(\overline{T})$ is indeed always true (Theorem~\ref{thm:spectral-inclusion}).
Naturally, one will asked next whether it then follows that $\spec(\overline{T})$ is the smallest closed cyclic subset of $\bbC$ that contains $\spec(T)$. 

In Section~\ref{sec:bigger-spectrum} we give a concrete counterexample to this question: 
if $T$ is a two-dimensional Jordan block with spectrum $\{1\}$, we show that $\spec(\overline{T})$ is the unit circle (Example~\ref{exa:counterexample-spectrum}).

Section~\ref{sec:functional-calculus} is an intermezzo on the multivariate functional calculus on Banach lattices for positive homogeneous functions. 
The construction of this functional calculus by means of the Kakutani representation theorem for AM-spaces with unit is explained on several occasions in the literature, but we could not find a reference that contains all the properties that we need in the sequel. 

As a first illustration of our usage of the functional calculus, we show in Section~\ref{sec:lattice-homomorphisms-non-semi-simple} how Example~\ref{exa:counterexample-spectrum} can be generalized to become a theorem for the spectrum of lattice homomorphisms.

In Section~\ref{sec:shape-of-spec-overline-T} we then come back to free Banach lattices and show that the point spectrum $\pntSpec(\overline{T})$ always contains a very specific sun-like shape and has quite strong algebraic closed properties (Theorems~\ref{thm:pnt-spec-sun-like} and~\ref{thm:pnt-spec-circle-subgroup}). 
As a consequence we next give a full description of $\spec(\overline{T})$ in the case when $T$ is essentially a root of unity (see below for a definition) in Theorem~\ref{thm:roots-of-unity-general} in Section~\ref{sec:spec-overline-ess-root}.

We turn to the approximate point spectrum in Section~\ref{sec:app-point-spec-overline-T} and show that if $\appSpec(T)$ contains two approximate eigenvalues $\lambda_1,\lambda_2$ such that $0<\modulus{\lambda_1}<\modulus{\lambda_2}$ then $\appSpec(\overline{T})$ already contains the ring of complex numbers $\lambda\in\bbC$ that satisfy $\modulus{\lambda_1}<\modulus{\lambda}<\modulus{\lambda_2}$ (Theorem~\ref{thm:dounut-from-eigenvalues}\ref{thm:dounut-from-eigenvalues:itm:approx}).

Finally in Section~\ref{sec:spec-char} we present a full characterization of the spectrum of $\overline{T}$ for a given bounded linear operator $T$ (Theorem~\ref{thm:spec-char-general}). To formulate this characterization we need the following terminology. 

We call a bounded linear operator $T$ on a complex Banach space $E$ a \emph{root of unity} if there is an $n\in\bbN$ such that $T^n$ is the identity operator on $E$. We say $T$ is \emph{essentially a root of unity} if there is an $n\in\bbN$ such that $T^n$ is a projection. Also we say $T$ is a \emph{proper essential root of unity} if $T$ is essentially a root of unity, but not a root of unity. 
We further characterize these properties in Lemma~\ref{lem:spectral-decomposition}\ref{lem:spectral-decomposition:itm:id} and Proposition~\ref{prop:hom-root-full-circle}. 

If $\spec(T) = \{0\}$ we know $\spec(\overline{T}) = \{0\}$ since the spectral radii of $T$ and $\overline{T}$ always satisfy $\spr(T) = \spr(\overline{T})$, so we only consider the case $\spr(T)>0$. 
Moreover, we assume that $\spr(T) = 1$, since otherwise we can simply rescale $T$. 
We use the notation 
\begin{align*}
    \spri(T)\coloneqq\inf\{\modulus{\lambda}\mid\lambda\in\spec(T)\setminus\{0\}\}
\end{align*}
as an ``inner analogue'' of the spectral radius. 
For further notation used in the following theorem, see Section~\ref{sec:prelim}.

Some finite dimensional examples that illustrate the following result can be seen in Figure~\ref{fig:spectra} at the end of the paper.

\begin{duplicate*}
    Let $E$ be a complex Banach space and  
    let $T \colon E \to E$ be a bounded linear operator with spectral radius $\spr(T) = 1$. 
    The spectrum of $\overline{T}$ can be described by the following cases. 
    \begin{center}
        \begin{adjustbox}{max width=\textwidth}
            \begin{tabu}{|c||c|c|}
                \hline
                & 
                \parbox[c][1.5cm]{0.38\textwidth}{$T$ is essentially a root of unity} & 
                \parbox[c][1.5cm]{0.38\textwidth}{$T$ is not essentially a root of unity} \\ 
                %\tabucline[1.5pt]{1-4}
                \hhline{|=||=|=|}
                \parbox[c][1.5cm]{0.14\textwidth}{$0\not\in\spec(T)$} & 
                \parbox[c][1.5cm]{0.38\textwidth}{$\spec(\overline{T})$ is the multiplicative subgroup of $\bbT$ generated by $\spec(T)$} &
                $\spec(\overline{T}) = [\spri(T),1]\,\bbT$ \\ 
                \cline{1-3}
                \multirow{2}{*}{\parbox[c][1.5cm]{0.14\textwidth}{$0\in\spec(T)$}} & 
                \multirow{2}{*}{\parbox[c][3cm]{0.38\textwidth}{$\spec(\overline{T})$ consists of $0$ and the multiplicative subgroup of $\bbT$ generated by $\spec(T)\setminus\{0\}$ }} & 
                \parbox[c][1.5cm]{0.38\textwidth}{if $0$ is pole of the resolvent of $T$: $\spec(\overline{T}) = [\spri(T),1]\,\bbT\cup\{0\}$} \\
                \cline{3-3}
                & & 
                \parbox[c][1.5cm]{0.38\textwidth}{if $0$ is not a pole of the resolvent of $T$: $\spec(\overline{T}) = \overline{\bbD}$} \\
                \hline
            \end{tabu}
        \end{adjustbox}
    \end{center}
    \vspace{0.25cm}
\end{duplicate*}

\section{Preliminaries}
\label{sec:prelim}

We denote the complex unit disc and unit circle by 
\begin{align*}
    \bbD = \{ z \in \bbC \mid \modulus{z} < 1 \}
    \quad \text{and} \quad
    \bbT = \partial\bbD = \{ z \in \bbC \mid \modulus{z} = 1 \}.
\end{align*}
Throughout the article we will freely use the theory of real and complex Banach lattices and positive operators acting on them. 
Standard references for this topic include the books  \cite{AliprantisBurkinshaw1985, MeyerNieberg1991, Schaefer1974}.
We will denote the spectrum of a bounded linear operator $T$ on a complex Banach space $E$ by $\spec(T)$, the point spectrum by $\pntSpec(T)$ and the approximate point spectrum by $\appSpec(T)$. 
If $\spec(T) \not= \{0\}$ then, additionally to the spectral radius $\spr(T)$, we define the inner spectral radius as $\spri(T)\coloneqq\inf\{\modulus{\lambda}\mid\lambda\in\spec(T)\setminus\{0\}\}$. Similarly we define $\pntSpr(T)$ and $\appSpec(T)$ for the (approximate) point spectrum. 
Finally, we also define $\pntSpri(T)$ and $\appSpri(T)$ in the same way as $\spri(T)$ if $\pntSpec(T)$ and $\appSpec(T)$ contain a non-zero element, respectively.
The identity operator on a space $E$ will be denoted by $\id_E$. 
By an \emph{operator} between two Banach space we always mean a linear operator throughout the article. 

Following the notation from \cite{deHeviaTradacete2023} we will, as mentioned above, denote the free complex Banach lattice over a complex Banach space $E$ by $\FBL_\bbC[E]$ and the isometrical embedding of $E$ into $\FBL_\bbC[E]$ by $\delta_E: E \to \FBL_\bbC[E]$. 

It follows from the universal property of the free Banach lattice, applied to the map $\delta_E T: E \to \FBL_\bbC[E]$, that there exists a unique lattice homomorphism $\overline{T}: \FBL_\bbC[E] \to \FBL_\bbC[E]$ that makes the diagram
\begin{center}
    \begin{tikzcd}
        \FBL_\bbC[E] \arrow[r, "\overline{T}"] & \FBL_\bbC[E] \\ 
        E \arrow[u, "\delta_E"] \arrow[r, "T"] & E \arrow[u, "\delta_E"]
    \end{tikzcd}
\end{center}
commute and has norm $\norm{\overline{T}} = \norm{\delta_E T} = \norm{T}$, since $\delta_E$ is isometric.
We will use the operator $\overline{T}$ throughout large parts of the article, and whenever we work with it, it is advisable to keep to above diagram in mind.

If $E$ is a complex Banach space and $T \colon E \to E$ is a bounded linear operator, 
then $0 \in \spec(T)$ if and only if $0 \in \spec(\overline{T})$; 
this was proved in \cite[Proposition~6.1(2)]{deHeviaTradacete2023} and we will often use this result tacitly throughout the article.

We will use the following representation result of complex free Banach lattices, which is a slight modification of a representation given in \cite{deHeviaTradacete2023}.
In \cite[Theorem~3.3 on p.\,6]{deHeviaTradacete2023} the free complex Banach lattice over a complex Banach space $E$ is characterized as $\FBL_\bbC[E]\coloneqq \FBL[E_\bbR] \oplus i \FBL[E_\bbR]$, which is a subspace of the space $C_{ph}((E_\bbR)',\bbC)$ endowed with the norm given in \cite[Definition~3.1 on p.\,5]{deHeviaTradacete2023}.
As shown in this reference, $\Psi: (E')_\bbR\to (E_\bbR)'$, $\Psi z' = \re z'$ is an isometric isomorphism of real Banach spaces. 
Now observe that the composition operator $\Phi:C_{ph}((E_\bbR)',\bbC)\to C_{ph}((E')_\bbR,\bbC)$, $\Phi f = f\circ\Psi$ is an isomorphism of Banach lattices. 
By defining $\tilde\delta_E:E\to \Phi \FBL_\bbC[E]$ as $\tilde\delta_E = \Phi\delta_E$ we can treat $\Phi \FBL_\bbC[E]\subseteq C_{ph}((E')_\bbR,\bbC)$ as the free complex Banach lattice over $E$. Therefore from here on we simply write $E'$ for $(E')_\bbR$ and $\delta_E$ instead of $\tilde\delta_E$. This makes a few calculations more intuitive as one can see from the proof of \cite[Lemma 3.2 on p.\,5]{deHeviaTradacete2023} that $\delta_E(z)(z') = \langle z',z\rangle$ for all $z\in E$ and $z'\in E'$.
So from now on we will use the following convention about the representation of free complex Banach lattices: 

\begin{convention}
    \label{conv:representation}
    Let $E$ be a complex Banach space. 
    We consider the free Banach lattice $\FBL_\bbC[E]$ as a continuously embedded subspace of $C_{ph}(E',\bbC)$. 
    More precisely, one has
    \begin{align*}
        \delta_E: 
        E \hookrightarrow \FBL_\bbC[E] \subseteq C_{ph}(E',\bbC)
    \end{align*}
    and the formula 
    \begin{align*}
        \delta_E(z)(z') = \langle z',z\rangle
    \end{align*} 
    holds for all $z\in E$ and $z'\in E'$. 
    Note that $\FBL_\bbC[E]$ need not be closed in $C_{ph}(E',\bbC)$ if $E$ is infinite-dimensional.
\end{convention}

We need a few results from spectral theory that we recall in the following lemmas. 

\begin{lemma} 
    \label{lem:spectral-decomposition}
    Let $E$ be a complex Banach space and $0 \not= T: E \to E$ a bounded linear operator. 
    \begin{enumerate}[label=\upshape(\alph*)]
        \item \label{lem:spectral-decomposition:itm:decomp}
        Let $T$ be a root of unity and choose $n \ge 1$ such that $T^n = \id_E$. 
        Then 
        \begin{align*}
            E 
            = 
            \oplus_{k=0}^{n-1} \ker(\lambda^k-T)
            ,
        \end{align*}
        where $\lambda \in \bbC$ is any primitive $n$-th root of unity.

        \item \label{lem:spectral-decomposition:itm:id}
        The operator $T$ is a root of unity if and only if $\spec(T)$ consists of finitely many roots of unity that are all first order poles of the resolvent. 
        In this case, $\spec(T) = \pntSpec(T)$ and the order of the subgroup of $\bbT$ generated by $\spec(T)$ is the smallest integer $n \ge 1$ such that $T^n = \id_E$.

        \item \label{lem:spectral-decomposition:itm:proj}
        Then $T$ is a proper essential roots of unity if and only if $\spec(T)$ consists of $0$, which is a pole of the resolvent of any order, and of finitely many roots of unity, which are first order poles of the resolvent. 
        In this case, $\spec(T) = \pntSpec(T)$.
    \end{enumerate}    
\end{lemma}

\begin{proof}
    All of this is classical spectral theory, so we only outline the main arguments. 
    \begin{description}[leftmargin=0.65cm]
        \item[\ref{lem:spectral-decomposition:itm:decomp}]
        The assertion of the lemma is a special case of the following classical result: 
        if $p$ is a polynomial function over $\bbC$ of degree $n$ whose roots $\lambda_0, \dots, \lambda_{n-1}$ are all mutually distinct, then 
        \begin{align*}
            \ker p(T)
            = 
            \oplus_{k=0}^{n-1} \ker(\lambda_k - T)
            ;
        \end{align*}
        a detailed proof of this can, for instance, be found in \cite[Lemma~2.2.3 on p.\,23]{Glueck2016}. 
        By applying this result to the polynomial function $p: z \mapsto z^n-1$ and using the assumption $p(T) = 0$, one gets the claim.
    
        \item[\ref{lem:spectral-decomposition:itm:id}]
        If $T$ is a root of unity, it follows from the decomposition in~\ref{lem:spectral-decomposition:itm:decomp} consists of finitely many roots of unity and that all of them are first order poles of the resolvent. 
        The decomposition from~\ref{lem:spectral-decomposition:itm:decomp} also implies the claims about the (point) spectrum of $T$. 
        
        Conversely, assume now that $\spec(T)$ consists of finitely many roots of unity that are all first order poles of the resolvent. 
        Then $E$ is the direct sum of the spectral spaces associated to the spectral values of $T$ and each spectral space is actually an eigenspace for the corresponding spectral value. 
        This readily implies that $T$ is a root of unity.

        \item[\ref{lem:spectral-decomposition:itm:proj}]
        Observe that, if $T$ is a proper essential root of unity, then it follows from the spectral mapping theorem for polynomials that $0$ is an isolated spectral value of $T$. 
        Hence, we may assume in any case that $0$ is an isolated spectral value of $T$. 
        Thus, we can split off the spectral value $0$ by using the spectral projection $Q \colon E \to E$ associated to $0$. 
        On the remainder of the space we can apply~\ref{lem:spectral-decomposition:itm:id}, 
        and on the range of $Q$ the restriction of $T$ has spectrum $0$ only -- hence, that part of the operator is essentially a root of unity if and only if it is nilpotent, and this is in turn equivalent to $0$ being a pole of the resolvent.
        \qedhere 
    \end{description}
\end{proof}

Finally we observe that, for a lattice lattice homomorphism $S$, very little spectral information suffices to characterize whether $S$ is essentially a root of unity:

\begin{proposition}
    \label{prop:hom-root-full-circle}
    Let $E$ be a complex Banach lattice and let $S \colon E \to E$ be a lattice homomorphism. 
    \begin{enumerate}[label=\upshape(\alph*)]
        \item\label{prop:hom-root-full-circle:itm:id} 
        Let $0 \not\in \spec(S)$. 
        The operator $S$ is a root of unity if and only if $\spec(S) \subsetneq \bbT$.

        \item\label{prop:hom-root-full-circle:itm:proj} 
        Let $0 \in \spec(S)$. 
        The operator $S$ is essentially a root of unity if and only if $\spec(S) \setminus \{0\} \subsetneq \bbT$.
    \end{enumerate}
\end{proposition}

\begin{proof}
    \begin{description}[leftmargin=0.65cm]
        \item[\ref{prop:hom-root-full-circle:itm:id}] 
        We may assume that $E \not= \{0\}$.
        If $S$ is a root of unity, then it follows from Lemma~\ref{lem:spectral-decomposition}\ref{lem:spectral-decomposition:itm:id} that $\spec(S) \subsetneq \bbT$. 
        So assume conversely that $\spec(S) \subsetneq \bbT$. 
        Since the spectrum of every lattice homomorphism is cyclic \cite[Theorem~V.4.4 on p.\,325]{Schaefer1974}, it follows that $\spec(S)$ consists of finitely many roots of unity. 
        Hence, there exists an integer $n \ge 1$ such that $\lambda^n = 1$ for all $\lambda\in\spec(\overline{T})$. 
        By the spectral mapping theorem for polynomials it follows that $\spec(S^n) = \{1\}$. 
        Thus, $S^n = \id_E$ is the identity operator by a spectral result for lattice homomorphisms \cite[Corollary 3.6(2) on p.\,204]{Arendt1983}. 

        \item[\ref{prop:hom-root-full-circle:itm:proj}] 
        If $S$ is essentially a root of unity, then $\spec(S) \setminus \{0\} \subsetneq \bbT$ due to Lemma~\ref{lem:spectral-decomposition}\ref{lem:spectral-decomposition:itm:proj}. 
        
        So assume now conversely that $\spec(S) \setminus \{0\} \subsetneq \bbT$. 
        Again by the cyclicity of $\spec(S)$ we conclude that $\spec(S) \setminus \{0\}$ consists of finitely many roots of unity, so there exists an integer $n \ge 1$ such that $\spec(S^n) \subseteq \{0,1\}$. 
        Since $S^n$ is a lattice homomorphism this implies, according to \cite[Theorem~3.5 on p.\,204]{Arendt1983}, that $S^n$ is an element of the so-called \emph{center} $Z(E)$ of $E$, which consists of all those bounded linear operators $T \colon E \to E$ for which there exists a number $c \ge 0$ such that $\modulus{Tx} \le c \modulus{x}$ for all $x \in E$. 

        It is known that $Z(E)$ a subalgebra of the bounded linear operators and a Banach lattice with respect to the operator norm and with respect to the usual order of operators on $E$. 
        Moreover, $Z(E)$ is isometrically isomorphic -- as a Banach lattice and as an algebra -- to $C(K,\bbC)$ for some compact Hausdorff space $K$; see for instance \cite[Theorem~3.1]{Arendt1983}. 
        For each $T \in Z(E)$ the spectrum of $T$ within $Z(E)$ coincides with the spectrum of $T$ within the space of bounded linear operators on $E$ \cite[Corollary~3.3]{Arendt1983}. 
        Those properties and the fact that $\spec(S^n) \subseteq \{0,1\}$ imply that $(S^n)^2 = S^n$, so $S^n$ is a projection, as claimed.
        \qedhere 
    \end{description}
\end{proof}

\section{Spectral inclusion}
\label{sec:spectral-inclusion}

The following theorem is our first main result and answers a question raised in \cite[before Proposition~6.4]{deHeviaTradacete2023}.

\begin{theorem}
    \label{thm:spectral-inclusion}
    Let $T: E \to E$ be a bounded linear operator on a complex Banach space $E$. Then the induced operator $\overline{T}$ on $\FBL_\bbC[E]$ satisfies $\spec(T) \subseteq \spec(\overline{T})$.
\end{theorem}

\begin{proof}
    Let $\mu\in\spec(T)$ and $\mu \not= 0$ 
    (for $\mu = 0$ the result follows from \cite[Proposition~6.1(2) on pp.\,19--20]{deHeviaTradacete2023}). 
    By rescaling $T$ we may assume that $\modulus{\mu} = 1$.
    If $\mu$ is in the approximate point spectrum $\appSpec(T)$, then it is also in $\appSpec(\overline{T})$ according to~\cite[Proposition~6.1(1) on pp.\,19--20]{deHeviaTradacete2023}. 
    So we may, and shall, assume for the rest of the proof that $\mu$ is not an approximate eigenvalue of $T$. 
    Then $\mu$ is in the topological interior of $\spec(T)$ and it is an eigenvalue of the dual operator $T'$, say with eigenvector $0 \not= x' \in E'$. 
    
    The universal property of $\FBLC[E]$ yields the existence of a functional $\widehat{x'} \in \big(\FBLC[E]\big)'$ such that $\widehat{x'}\delta_E = x'$ and $\norm{\widehat{x'}} = \norm{x'}$. 
    We have
    \begin{equation} 
        \label{eq:eig-mu-id}
        \overline{T}' \, \widehat{x'}
        = 
        \widehat{x'}\;\overline{T} 
        = 
        \widehat{x'T} 
        = 
        \widehat{T'x'} 
        = 
        \widehat{\mu x'}
        =
        \widehat{x' \, \mu \id_E}
        = 
        \widehat{x'} \; \overline{\mu\id_E}
        = 
        \overline{\mu\id_E}' \, \widehat{x'}
        .
    \end{equation}
    Now, we first consider the special case that $\mu$ is a primitive $p$-th root of unity for a prime number $p \in \bbN$. 
    Then $\overline{\mu\id_E}^p = \overline{\mu^p\id_E} = \overline{\id_E} = \id_{\FBLC[E]}$. 
    Hence, one has the decomposition 
    \begin{align*}
        \FBLC[E] 
        = 
        \oplus_{k=0}^{p-1} \ker\big(\mu^k-\overline{\mu\id_E}\big)
    \end{align*}
    according to Lemma~\ref{lem:spectral-decomposition}.
    For each $k \in \{0,\dots,p-1\}$ let $P_k: \FBLC[E] \to \FBLC[E]$ denote the projection onto the $k$-th component of this decomposition. 
    Then $\overline{\mu\id} = \sum_{k=0}^{n-1}\mu^k P_k$.
    Moreover, since $\overline{T}$ commutes with $\overline{\mu\id_E}$, it also commutes with the projections $P_0, \dots, P_{p-1}$; 
    this follows since those projections are spectral projections of $\overline{\mu\id_E}$ and can hence be written as integrals over the resolvent of $\overline{\mu\id_E}$.
    
    Also note that, for each $k \in \{0,\dots,p-1\}$, the dual operator $P_k'$ is a projection onto $\ker\big(\mu^k-\overline{\mu\id_E}'\big)$.
    So it follows from~\eqref{eq:eig-mu-id} that
    \[ 
        \overline{T}'P_k'\,\widehat{x'} 
        = 
        P_k'\,\overline{T}'\,\widehat{x'} 
        = 
        P_k'\,\overline{\mu\id_E}'\,\widehat{x'} 
        = 
        \overline{\mu\id_E}' \, P_k' \, \widehat{x'} 
        = 
        \mu^kP_k'\,\widehat{x'} 
    \]
    for each $k \in \{0, \dots, p-1\}$.
    
    We now show that there exists at least one integer $k_0 \in \{1, \dots,p-1\}$ such that $P_{k_0}' \widehat{x'} \not= 0$ and hence, $\mu^{k_0}$ is an eigenvalue of $\widehat{T}'$ by the previous equality. 
    So assume to the contrary that $P_{k}' \widehat{x'} = 0$ for all $k \in \{1, \dots, p-1\}$. 
    Then $\widehat{x'} = P_0' \widehat{x'}$, so $\widehat{x'}$ is a fixed vector of $\overline{\mu\id_E}'$ and hence,
    \[ 
        \widehat{\mu x'} 
        = 
        \widehat{x'} \; \overline{\mu\id_E} 
        = 
        \overline{\mu\id_E}' \, \widehat{x'} 
        = 
        \widehat{x'}
        ,
    \]
    which implies $\mu x' = x'$. 
    This is a contradiction since $x' \not= 0$ and $\mu \not= 1$ 
    (as $\mu$ is assumed to be a primitive $p$-th root of unity). 
    
    So we have shown that $\mu^{k_0}$ is an eigenvalue of $\overline{T}'$ and thus a spectral value of $\overline{T}$ for at least one integer $k_0 \in \{1, \dots,p-1\}$.
    Since the spectrum of a lattice homomorphism is cyclic \cite[Theorem~V.4.4 on p.\,325]{Schaefer1974}, all integer powers of $\mu^{k_0}$ are spectral values of $\overline{T}$, too. 
    As $p$ is prime and $k_0 \in \{1,\dots, p-1\}$, the number $\mu^{k_0}$ is also a primitive $p$-th root of unity and hence, there exists an integer $j \ge 1$ such that $\mu = (\mu^{k_0})^j$. 
    Hence, $\mu$ is itself a spectral value of $\overline{T}$.

    Finally, consider the case where $\mu$ is a general point of modulus $1$ in the interior of $\spec(T)$. 
    Then there exists a sequence $(\mu_n)$ in $\spec(T)$ that converges to $\mu$ and such that each $\mu_n$ is a primitive root of unity of a prime order. 
    By what we have just shown one has $\mu_n \in \spec(\overline{T})$ for each index $n$ 
    and since the spectrum is closed, it follows that $\mu\in\spec(\overline T)$.
\end{proof}

We end this section with a slightly tangential, but still worthwhile observation. 
The following proposition was shown in \cite[Proposition~6.5 on p.\,21]{deHeviaTradacete2023}. We demonstrate that its proof becomes very simple if one uses Lemma~\ref{lem:spectral-decomposition}.

\begin{proposition} 
    \label{prop:spectrum-characterisation}
    Let $E$ be a complex Banach space and let $\mu\in\bbC$ be a primitive $n$-th root of unity for an integer $n \ge 1$. 
    Then 
    \[
        \spec(\overline{\mu\id_E}) 
        = 
        \big\{\mu^k\mid k \in \{0,\dots,n-1 \}\big\}.
    \]
\end{proposition}

\begin{proof}
    As $\mu$ is an eigenvalue of $\mu\id_E$, it easily follows that $\mu$ is also an eigenvalue of $\overline{\mu\id_E}$, see \cite[Proposition~6.1(a)]{deHeviaTradacete2023} for the details. 
    Since the point spectrum of a lattice homomorphism is cyclic \cite[Corollary~2 to Proposition~V.4.2 on p.\,324]{Schaefer1974} we conclude that
    \[ 
        \{\mu^k\mid k=0,\dots,n-1\} 
        \subseteq 
        \spec(\overline{\mu\id_E}).
    \]
    The converse inclusion follows from Lemma~\ref{lem:spectral-decomposition} since 
    $\big(\overline{\mu \id_E}\big)^n = \overline{\mu^n \id_E} = \overline{\id_E} = \id_{\FBLC[E]}$. 
\end{proof}

\section{A counterexample}
\label{sec:bigger-spectrum}

As the spectrum of a lattice homomorphism is always cyclic, Theorem~\ref{thm:spectral-inclusion} raises the question if $\spec(\overline T)$ is the smallest closed cyclic subset of $\bbC$ that contains $\spec(T)$. 
Proposition~\ref{prop:circle} and Example~\ref{exa:counterexample-spectrum} below show that this is not the case, in general. 
We start with the following simple lemma that we will need several times throughout the article.

\begin{lemma} \label{lem:spectral-radius}
    Let $T:E\to E$ be a bounded linear operator on a Banach space $E$. Then $\spec(\overline{T})\subseteq [r,\spr(T)]\bbT$ where $r\coloneqq \inf_{\lambda\in\spec(T)}\modulus{\lambda}$.
\end{lemma}

\begin{proof}
    Observe that $T$ and $\overline{T}$ have the same spectral radius.
    This follows from Gelfand's spectral radius formula, see \cite[Proposition~6.1(3)]{deHeviaTradacete2023} for details. 
    If $0 \in \spec(T)$, then $r = 0$ and the lemma is proved. 

    Assume now that $0 \not\in \spec(T)$. 
    Then $T$ is invertible, hence so is $\overline{T}$, and $\overline{T}^{-1} = \overline{T^{-1}}$; this follows right from the definition of $\overline{T}$ and can also be found in \cite[Proposition~6.1(2) and its proof]{deHeviaTradacete2023}. 
    Hence, one also has $\spr(\overline{T}^{-1}) = \spr(T^{-1}) = \frac1{r}$, 
    which implies that all spectral values of $\overline{T}$ have modulus $\ge r$, as claimed.
\end{proof}

\begin{proposition}
    \label{prop:circle}
    Let $E$ be a complex Banach space and $T: E \to E$ a bounded linear operator that satisfies $\spec(T) \subseteq \bbT$. 
    Then $\spec(\overline{T}) \subseteq \bbT$ and precisely one of the following two assertions holds:
    \begin{enumerate}[label=\upshape(\arabic*)]
        \item\label{prop:circle:itm:power} 
        The operator $T$ is a root of unity. 

        \item\label{prop:circle:itm:spec} 
        One has $\spec(\overline{T}) = \bbT$.
    \end{enumerate}
\end{proposition}

\begin{proof}
    Note that $T$ is a root of unity if and only if $\overline{T}$ is a root of unity. 
    Hence, the proposition readily follows from Proposition~\ref{prop:hom-root-full-circle}\ref{prop:hom-root-full-circle:itm:id}.
\end{proof}

If one applies Proposition~\ref{prop:circle} to the Jordan block $T=\tmatrix{1&1\\ 0&1}$ that acts on $\bbC^2$, one obtains $\spec(T) = \{1\}$ but $\spec(\overline{T}) = \bbT$; 
so the spectrum of $\overline{T}$ is not the smallest closed cyclic set that contains $\spec(T)$, in general.
For this simple example it is natural to ask whether the spectral values of $\spec(\overline{T})$ are even eigenvalues and if one can explicitly construct eigenvectors for them.
We show now that this is indeed possible.

\begin{example}
    \label{exa:counterexample-spectrum}
    On the complex Banach space $\bbC^2$ consider the operator $T:\bbC^2\to\bbC^2$ given by $T=\tmatrix{1&1\\ 0&1}$. 
    Then $\spec(\overline{T}) = \pntSpec(\overline{T}) = \bbT$.
\end{example}

\begin{proof}
    From Proposition~\ref{prop:circle} we know that $\spec(\overline{T}) = \bbT$, so we only need to show that every number in $\bbT$ is an eigenvalue of $\overline{T}$.
    
    For the number $1$ this is easy:
    Since $(1,0)^\mathsf{T}$ is an eigenvector of $T$ for the eigenvalue $1$, one readily obtains that the point evaluation $\delta_{\bbC^2}\big( (1,0)^\mathsf{T} \big)$ is an eigenvector of $\overline{T}$ for the eigenvalue $1$.
    For all other numbers in $\bbT$ finding an eigenvector is more involved.
    To this end we will use the representation of $\FBL_\bbC[E]$ described in Convention~\ref{conv:representation}.

    According to this convention, the free Banach lattice $\FBL_\bbC[\bbC^2]$ is embedded into $H\big( (\bbC^2)', \bbC \big)$. 
    Since $\bbC^2$ is finite dimensional, one actually has $\FBL_\bbC[\bbC^2] = H\big( (\bbC^2)', \bbC \big)$. 
    This can, for instance, be seen by using that the free real Banach lattice $\FBL[\bbR^4]$ is isomorphic to the free real Banach lattice generated by a set with four elements \cite[Corollary 2.8 on p.\,5]{AvilesRodriguezTradacete2018}, which can in turn be represented as a space of homogeneous functions in four real variables \cite[Proposition 5.3 on p.\,10]{PagterWickstreadAnthony2015}. 

    The operator $\overline{T}$ acts on $H\big( (\bbC^2)', \bbC \big)$ as the composition operator with $T' = \tmatrix{1&0\\ 1&1}$, i.e.\ one has $\overline{T}f = f \circ T'$ for all $f \in H\big( (\bbC^2)', \bbC \big)$ (see \cite[Lemma 3.1 on p.\,27]{OikhbergTaylorTradaceteTroitsky2022}). Now fix a number $e^{i\varphi} \in \bbT$ with $\varphi \in \bbR$.
    Let $f\in H\big( (\bbC^2)', \bbC \big)$ be given by 
    \begin{align*}
        f(z') 
        = 
        \begin{cases}
            z_1'\exp(i\varphi\re\tfrac{z_2'}{z_1'}), \quad & \text{if } z_1' \not= 0, \\ 
            0, \quad & \text{if } z_1' = 0
        \end{cases} 
    \end{align*}
    for all $z' \in (\bbC^2)' \simeq \bbC^2$. 
    Note that $f$ is indeed continuous and positively homogeneous.
    We claim that $f$ is an eigenfunction of $\overline{T}$ for the eigenvalue $e^{i\varphi}$. 
    To see this, let $z' \in (\bbC^2)' \simeq \bbC^2$. 
    One has $(T'z')_1 = z_1'$ and $(T'z')_2 = z_1' + z_2'$. 
    So if $z_1' = 0$, then $(\overline{T}f)(z') = f(T'z') = 0 = e^{i\varphi}f(z')$. 
    If, on the other hand, $z_1' \not= 0$, then
    \begin{align*}
        (\overline{T}f)(z') 
        = 
        f(T'z') 
        = z_1'\exp(i\varphi\re\tfrac{z_1'+z_2'}{z_1'}) = e^{i\varphi}f(z')
    \end{align*}
    So we indeed have $\overline{T}f = e^{i\varphi}f$, as claimed.
\end{proof}

\begin{remark}
    Proposition~\ref{prop:circle} can be used to see that the map $T\mapsto \overline{T}$ is not continuous. Clearly the sequence $(T_n) = \left(\tmatrix{1&\frac1n\\0&1}\right)$ converges to the identity. 
    But for each $n$ one has $-1\in \spec(\overline{T_n})$ according to Proposition~\ref{prop:circle}, and hence $\norm{\overline{T_n}-\id_{FBL_\bbC[\bbC^2]}} \geq \spr\big(\overline{T_n}-\id_{FBL_\bbC[\bbC^2]}\big) \ge 2$.
\end{remark}

\section{The functional calculus on Banach lattices revisited}
\label{sec:functional-calculus}

We would like to use similar constructions of eigenvectors as in the proof of Example~\ref{exa:counterexample-spectrum} in more general situations -- in particular on free Banach lattices $\FBL_\bbC[E]$ that cannot be represented as a full space of homogeneous functions, which can happen for infinite-dimensional $E$. 
It turns out that essentially the same construction still works, though, if one uses a more abstract approach: 
the functional calculus on Banach lattices for continuous, positively homogeneous functions. The existence of such a functional calculus is well-known, see for instance \cite[Theorem 1.3 on p.\,3]{LaustenTroitsky2020} or \cite[Theorem 3.7 on p.\,429]{BuskesDePagterVanRooij1991}. 
However, none of the references seem to contain all properties that we will need in the subsequent sections, in particular the fact the the functional calculus commutes with lattice homomorphisms and that it is uniformly continuous on bounded sets. 

In order to prove those properties and to have all results that we need available in one place, a somewhat self-contained treatment seems appropriate. 
We present such a treatment of the functional calculus in this section. 
To provide a good source of reference also in the future, we treat both the case of real and complex scalars simultaneously, even though we only need the complex case in the rest of the paper.

Let us first fix some notation. 
Let $\bbK \in \{\bbR,\bbC\}$.
Let $E$ be a Banach lattice over the field $\bbK$, and $m,n\in \bbN$. 
We endow the space $C_{ph}(\bbK^n,\bbK^m)$ of continuous, positively homogeneous functions from $\bbK^n$ to $\bbK^m$ with the norm
\begin{align*}
    \norm{f} 
    \coloneqq 
    \sup_{\norm{z}_1 \leq 1}\norm{f(z)}_\infty 
    = 
    \sup_{\norm{z}_1 = 1}\norm{f(z)}_\infty
    .
\end{align*}
Moreover, we endow the real vector space $C_{ph}(\bbK^n,\bbR^m)$ with the pointwise order. 
This renders $C_{ph}(\bbK^n,\bbK^m)$ a Banach lattice over $\bbK$.

The space $E^n$ is also a Banach lattice over $\bbK$ with the pointwise lattice operations and any of the two norms $\norm{\cdot}_1$ or $\norm{\cdot}_\infty$ defined by 
\begin{align*}
    \norm{x}_1 \coloneqq\sum_{k=1}^n \norm{x_k}_E 
    \qquad \text{and} \qquad 
    \norm{x}_\infty \coloneqq \max_{k=1,\dots,n}\norm{x_k}_E
\end{align*}
for all $x \in E^n$. 

One wants to define a functional calculus in the following sense. 
For every $f \in C_{ph}(\bbK^n,\bbK^m)$ and every $x \in E^n$, we would like to define an element $f(x) \in E^m$. 
This should at least have the following properties: 
\begin{enumerate}[label=(\alph*)]
    \item 
    If $E$ is a function space and $n=m=1$, then the functional calculus should just mean composition, i.e\ $f(x) = f \circ x$ for every $x \in E$. 

    More generally, if $E$ is a function space and $m,n \in \bbN$, then $E^n$ and $E^m$ can be identified with spaces of $\bbK^n$- and $\bbK^m$-valued functions, so in this case the functional calculus should also be the composition, i.e.\ again $f(x) = f \circ x$ for every $x \in E^n$. 

    This shows in particular that for each fixed $x \in E$ the mapping $f \mapsto f(x)$ should be a lattice homomorphism.

    \item 
    The most simple functions in the space $C_{ph}(\bbK^n,\bbK^m)$ are the linear ones. 
    So consider a matrix $A=(A_{jk}) \in \bbK^{m\times n}$, which we identify with a linear map $A \in C_{ph}(\bbK^n,\bbK^m)$. 
    The matrix $A$ also induces a map $A_E: E^n \to E^m$ in a canonical way, by setting
    \begin{align}
        \label{eq:linear-function-calculus}
        A_E x 
        := 
        \begin{pmatrix}
            \sum_{k=1}^n A_{1k}x_k \\ 
            \vdots \\ 
            \sum_{k=1}^n A_{mk}x_k
        \end{pmatrix}
    \end{align}
    for all $x \in E^n$. 
    It is a natural requirement that the functional calculus should be consistent with this way of extending linear functions, i.e.\ that $A(x) = A_E x$ should hold for all $x \in E^n$, where the left hand side $A(x)$ denotes the vector obtained from the functional calculus by using the function $A \in C_{ph}(\bbK^n,\bbK^m)$.
\end{enumerate}

The following theorem shows that a functional calculus with properties~(a) and~(b) exists and that it is, in fact, uniquely determined by those properties.

\begin{theorem}[Existence and uniqueness of the functional calculus] 
    \label{functionalCalculus}
    Let $E$ be a Banach lattice over $\bbK\in\{\bbR,\bbC\}$ and $m,n\in\bbN$. 
    For every $x\in E^n$ there is exactly one lattice homomorphism
    \[
        \Phi_x: C_{ph}(\bbK^n,\bbK^m) \to E^m 
    \]
    that satisfies $\Phi_x(A) = A_E x$ for all $A\in \bbK^{m\times n}$, where $A_E x$ is defined by formula~\eqref{eq:linear-function-calculus}.
\end{theorem}

\begin{proof}   
    Fix $x \in E^n$.

    We first show the uniqueness part. 
    Let $V$ denote the Banach sublattice of $C_{ph}(\bbK^n,\bbK^m)$ generated by the linear maps. 
    We show that that $V$ is dense in $C_{ph}(\bbK^n,\bbK^m)$; 
    uniqueness of $\Phi_x$ is then clear since $\Phi_x$ is a lattice homomorphism and is continuous. 

    To show the density of $V$ we use the Stone--Weierstraß approximation theorem. 
    Let $S^{n-1}$ denote the unit sphere in $\bbK^n$ with respect to the norm $\norm{\argument}_1$. 
    For $M \coloneqq \dot\cup_{k=1}^m S^{n-1}$ the linear map 
        \[ J:C_{ph}(\bbK^n,\bbK^m)\to C(M,\bbK) \]
    defined by $(Jf)(z) = (f(z))_k$ for all $z\in S^{n-1}_k$, all $f\in C_{ph}(\bbK^n,\bbK^m)$, and all $k \in \{1, \dots, m\}$, is an isometric isomorphism of Banach lattices; 
    so it suffices to show that $JV$ is dense in $C(M;\bbK)$.
    For each $k \in \{ 1,\dots,n \}$ one has $p_k\in V$, where $p_k(z)\coloneqq (z_k,\dots,z_k)$ for all $z\in \bbK^n$. 
    We define $p\in V$ by
        \[ p(z) \coloneqq \sum_{k=1}^n|p_k|(z) = (\norm{z}_1,\dots,\norm{z}_1). \]
    This yields $Jp = \mathbbm{1}$. As $JV$ separates the points of $M$, contains $\mathbbm{1}$ and is closed under complex conjugation in the case $\bbK = \bbC$, the Stone-Weierstrass theorem yields that $JV$ is dense in $C(M,\bbK)$, and hence $V$ is dense in $C_{ph}(\bbK^n,\bbK^m)$.

    We now show the existence of $\Phi_x$. 
    Choose some $y\geq x_k$ for all $k=1,\dots,n$. Then $x\in (I_y)^n\subseteq E^n$ where $I_y \subseteq E$ is the principal ideal generated by $y$ in $E$. 
    The Kakutani representation theorem says there is an isomorphism of Banach lattices $j:I_y\to C(K,\bbK)$ for some compact Hausdorff space $K$, and $j$ clearly extends to Banach lattice isomorphisms $j_n \colon (I_y)^n \to C(K; \bbK^n)$ and $j_m \colon (I_y)^m \to C(K; \bbK^m)$.
    We now define $\Phi_x$ by
    \begin{equation} \label{PhiDarstellung}
        \Phi_x(f)\coloneqq j_m^{-1}(f\circ j_n(x))
    \end{equation} 
    for all $f\in C_{ph}(\bbK^n,\bbK^m)$. 
    As all positive operators on Banach lattices are bounded, so is $\Phi_x$. The identity $\Phi_x(A) = A_Ex$ for all $A\in \bbK^{m\times n}$ follows immediately from the linearity of $j_n$ and $j_m$.
\end{proof}

\begin{definition}[Notation for the functional calculus]
    Let $E$ be a Banach lattice over $\bbK\in\{\bbR,\bbC\}$ and $m,n\in\bbN$. 
    For every $f \in C_{ph}(\bbK^n,\bbK^m)$ and every $x \in E^n$ we set 
    \begin{align*}
        f(x) := \Phi_x(f) \in E^m
    \end{align*}
    where $\Phi_x: C_{ph}(\bbK^n,\bbK^m) \to E^m$ denotes the lattice homomorphism from Theorem~\ref{functionalCalculus}.
\end{definition}

\begin{remark}
    The uniqueness part of Theorem~\ref{functionalCalculus} implies that the representation of $\Phi_x$ in \eqref{PhiDarstellung} does not depend on $y$.
\end{remark}

If $F$ is also a Banach lattice over $\bbK$ and $S: E \to F$ is a bounded linear operator, then we use the notation 
\begin{align}
    \label{eq:extend-operator-to-power}
    Sx
    = 
    S 
    \begin{pmatrix}
        x_1 \\ 
        \vdots \\ 
        x_n
    \end{pmatrix}
    := 
    \begin{pmatrix}
        Sx_1 \\ 
        \vdots \\ 
        Sx_n
    \end{pmatrix}
    \in 
    F^n
\end{align}
for each $x \in E^n$.

\begin{theorem}[Properties of the functional calculus]
    \label{thm:funccalc-properties}
    Let $E$ be a Banach lattice over $\bbK\in\{\bbR,\bbC\}$ and $\ell,m,n\in\bbN$. 
    Let $f \in C_{ph}(\bbK^n,\bbK^m)$. 
    The functional calculus has the following properties. 
    \begin{enumerate}[label=\upshape(\alph*)]
        \item\label{thm:funccalc-properties:itm:modulus} 
        If $h \in C_{ph}(\bbK^n,\bbK^n)$ is the componentwise modulus function, then $h(x) = \modulus{x}$ is the componentwise modulus of $x$ for every $x \in E^n$. 
        The same holds for the real, imaginary, positive and negative part.

        \item\label{thm:funccalc-properties:itm:assosiative} 
        Let $g\in C_{ph}(\bbK^m,\bbK^\ell)$. For all $x\in E^n$ we have
            \[ (g\circ f)(x) = g(f(x)). \]

        \item\label{thm:funccalc-properties:itm:LatHom} 
        For every Banach lattice homomorphism $S:E\to F$ we have $f(Sx) = Sf(x)$ for all $x\in E^n$, where we used the notation defined in formula~\eqref{eq:extend-operator-to-power}.
        
        \item\label{thm:funccalc-properties:itm:positive2} 
        Assume that $f(z) \neq 0$ for all $z \neq 0$. Then $f(x) \neq 0$ for all $x\in E^n\setminus\{0\}$.

        \item\label{thm:funccalc-properties:itm:positive} 
        If $f\in C_{ph}(\bbK^n,\bbK^m)$ satisfies $f \ge 0$, then $f(x)\geq 0$ for all $x\in E^n$. 
        Hence, if $f,g\in C_{ph}(\bbK^n,\bbK^m)$ satisfy $f \le g$, then $f(x)\leq g(x)$ for all $x\in E^n$.

        \item\label{thm:funccalc-properties:itm:monotone}
        If $f$ is increasing on $\bbK^n_+$ in the sense that, for all $z_1,z_2\in \bbK^n$, the inequality $0 \le z \leq \tilde z$ implies $f(z) \leq f(\tilde z)$, then $0\leq x\leq y $ implies $f(x) \leq f(y)$ for all $x,y\in E^n$.

        \item\label{thm:funccalc-properties:itm:sublattice}
        Every closed vector sublattice $F$ of $E$ is invariant under the positive homogeneous functional calculus. That is, for $x\in F^n$ we have $f(x) \in F^m$.

        \item\label{thm:funccalc-properties:itm:bounded}
        For every $x\in E$ one has $\norm{f(x)}_\infty \leq \norm{f}\norm{x}_1$.

        \item\label{thm:funccalc-properties:itm:ideals}
        Let $I$ be a closed ideal in $E$. Then $x-y\in I^n$ implies $f(x)-f(y)\in I^m$ for all $x,y\in E^n$.
        
        \item\label{thm:funccalc-properties:itm:continuous}
        For all bounded sequences $(x_k),(y_k)\subseteq E^n$ the following holds: 
        if $x_k-y_k \to 0$ as $k \to \infty$, 
        then $f(x_k)-f(y_k) \to 0$ as $k \to \infty$.
    \end{enumerate}
\end{theorem}

\begin{proof}
    \begin{description}[leftmargin=0.65cm]
        \item[\ref{thm:funccalc-properties:itm:modulus}, \ref{thm:funccalc-properties:itm:assosiative} and~\ref{thm:funccalc-properties:itm:monotone}] 
        These properties can be shown by simply applying \eqref{PhiDarstellung}.

        \item[\ref{thm:funccalc-properties:itm:LatHom}] 
        Let $x\in E^n$ and $y\in E$ such that $y\geq |x_k|$ for all $k=1,\dots,n$. The operator $S|_{I_y}:I_y\to I_{Sy}$ is a Banach lattice homomorphism and the Kakutani representation theorem gives isomorphisms $j:I_y\to C(K,\bbK)$ and $\tilde j:I_{S_y}\to C(\tilde K,\bbK)$ for some compact Hausdorff spaces $K$, $\tilde K$ such that $jy = \one_K$ and $\tilde j Sy = \one_{\tilde K}$. 
        Thus 
            \[ T := \tilde jSj^{-1} : C(K,\bbK)\to C(\tilde K,\bbK) \]
        is also a lattice homomorphism and it satisfies 
        $T\one_K = \tilde j Sy = \one_{\tilde K}$. 
        By \cite[Example~2.2 on p.\,200]{Arendt1983} there exists a continuous function $\varphi:\tilde K\to K$ such that
            \[ Tg = g\circ \varphi \]
        for all $g\in C(K,\bbK)$. 
        For all $f\in C_{ph}(\bbK^n,\bbK^m)$ and all $g\in C(K,\bbK^n) = C(K,\bbK)^n$ we thus have
            \begin{equation*} 
                f \circ (Tg) = f\circ (g \circ \varphi) = (f \circ g) \circ \varphi = T(f\circ g),
            \end{equation*} 
        where we again used the notation from formula~\eqref{eq:extend-operator-to-power}.
        From \eqref{PhiDarstellung} and the definition of $T$ we now get
            \begin{align*}
                f(Sx) 
                &= 
                \tilde j_m^{-1}(f\circ \tilde j_n(Sx)) 
                = 
                \tilde j_m^{-1}(f\circ Tj_n(x)) 
                \\ 
                &= 
                \tilde j_m^{-1} T(f \circ j_n(x)) 
                = 
                S j_m^{-1} (f \circ j_n(x)) = Sf(x)
                .
            \end{align*}
        
        \item[\ref{thm:funccalc-properties:itm:positive2}] 
        Let $x\neq 0$. Using \eqref{PhiDarstellung} we have that $j_n(x)\neq 0$ because $j_n$ is a lattice isomorphism. Thus $f\circ j(x) \neq 0$ and as $j_m^{-1}$ is a lattice isomorphism we also get $f(x) = j_m^{-1}(f\circ j_n(x))\neq 0$.

        \item[\ref{thm:funccalc-properties:itm:positive}] 
        From~\ref{thm:funccalc-properties:itm:modulus} we get $f(x) = \modulus{f}(x) = \modulus{f(x)} \geq 0$.
        If $f\leq g$ it follows that $(g-f)(x) \geq 0$.

        \item[\ref{thm:funccalc-properties:itm:sublattice}] 
        Let $x \in E^n$ and let $J:F\hookrightarrow E$ denote the canonical embedding. 
        Let $f_E(x)$ denote the functional calculus computed in the Banach lattice $E$, and likewise for $f_F(x)$. 
        Since $J$ is a lattice homomorphism, it follows from~\ref{thm:funccalc-properties:itm:LatHom} that $f_E(x) = f_E(Jx) = Jf_F(x) \in J(F) = F$.

        \item[\ref{thm:funccalc-properties:itm:bounded}] 
        First we make the following observation about notational ambiguity: 
        for each $x \in E^n$ the notation $\norm{x}_1$ can either denote the number $\norm{x_1}_E + \dots + \norm{x_n}_E \in [0,\infty) \subseteq \bbK$ or an element of $E$ that is defined in terms of the functional calculus if we interpret the notation $\norm{\argument}_1$ as the $1$-norm on $\bbK^n$ and thus as an element of $C_{ph}(\bbK^n,\bbK)$.
        There we use the notation $g\coloneqq \norm{\argument}_1\in C_{ph}(\bbK^n,\bbK)$ in this proof and hence write $g(x)$ for the functional calculus meaning of $\norm{x}_1$.
        
        For each $k=1,\dots,m$ and all $z \in \bbK^n$ we have 
            \[ \modulus{f_k(z)} = \modulus{f_k\left(\tfrac{z}{\norm{z}_1}\right)}\norm{z}_1\leq \norm{f}\norm{z}_1 = \norm{f}g(z). \] 
        Thus~\ref{thm:funccalc-properties:itm:positive} implies $\modulus{f_k(x)} \leq \modulus{\norm{f}g(x)} \in E$ for all $x\in E^n$. As $E$ is a Banach lattice this implies $\norm{f_k(x)}_E \leq \norm{f} \norm{g(x)}_E \leq \norm{f} \norm{x}_1$. From here it is also clear that $\norm{f(x)}_\infty \leq \norm{f}\norm{x}_1$.

        \item[\ref{thm:funccalc-properties:itm:ideals}] 
        The quotient map $Q: E\to E/I$ is a lattice homomorphism. 
        So if $x-y \in I^n$, it follows from $Qx = Qy$ and~\ref{thm:funccalc-properties:itm:LatHom} that
            \[ Qf(x) = f(Qx) = f(Qy) = Qf(y), \]
        so $f(x) - f(y) \in I^m$.

        \item[\ref{thm:funccalc-properties:itm:continuous}] 
        Let $x = (x_k)$ and $y = (y_k)$ be bounded sequences in $E^n$ that satisfy $x_k-y_k \to \infty$ as $k \to \infty$.
        For each $k \in \bbN$ the $k$-th coordinate map $\ell^\infty(\bbN,E) \to E$, $w \mapsto w_k$ is a lattice homorphism, 
        so $(f(v))_k = f(v_k)$ for each $v \in \ell^\infty(\bbN,E)^n$ and each $k \in \bbN$. 
        
        Now consider $x,y$ as elements of $\ell^\infty(\bbN,E)^n$. 
        Then $x-y \in c_0(\bbN,E)^n$. 
        Since $c_0(\bbN,E)$ is a closed ideal in $\ell^\infty(\bbN,E)$, 
        it follows from~\ref{thm:funccalc-properties:itm:ideals} that 
        $f(x) - f(y) \in c_0(\bbN,E)^m$. 
        Thus, $f(x_k) - f(y_k) = (f(x) - f(y))_k \to 0$ as $k \to \infty$.
        \qedhere
    \end{description}
\end{proof}

One can rephrase the assertion of Theorem~\ref{thm:funccalc-properties}\ref{thm:funccalc-properties:itm:continuous} in the following way.

\begin{corollary}
    Let $E$ be a Banach lattice over $\bbK\in\{\bbR,\bbC\}$ and $f\in C_{ph}(\bbK^n,\bbK^m)$. The map $f:E^n\to E^m,\, x\mapsto f(x) = \Phi_x(f)$ induced by the positive homogeneous functional calculus is uniformly continuous on bounded sets on $E^n$.
\end{corollary}

\begin{proof}
    Assume that there exists a bounded subset $B \subseteq E^n$ such that $f$ is not uniformly continuous on $B$.
    Then there is an $\varepsilon>0$ such that for every $k\in \bbN$ there are $x_k,y_k\in B$ such that $\norm{x_k-y_k} \le \frac1k$ and $\norm{f(x_k)-f(y_k)} \ge \varepsilon$. 
    But this contradicts Theorem~\ref{thm:funccalc-properties}\ref{thm:funccalc-properties:itm:continuous}.
\end{proof}

We close the section with a lemma that illustrates how order related properties of a function translate to similar properties of the functional calculus.

\begin{lemma} \label{lem:func-cal-nonzero}
    Let $E$ be a Banach lattice over $\bbK\in\{\bbR,\bbC\}$ and $f\in C_{ph}(\bbK^n,\bbK^m)$ such that, for all $z \in \bbK^n$, the condition $\modulus{z_1} \land \dots \land \modulus{z_n} \neq 0$ implies $f(z)\neq 0$. 
    Then $f(x)\neq 0$ for all $x\in E^n$ that satisfy $x_1\land\dots\land x_n \neq 0$.
\end{lemma}

\begin{proof}
    Using \eqref{PhiDarstellung} for $f(x)$ we see that $\modulus{j(x_1)}\land\dots\land\modulus{j(x_n)} \neq 0$ as $\modulus{x_1} \land\dots\land \modulus{x_n} \neq 0$. Thus $f\circ j_n(x) \neq 0$ and therefore $f(x)\neq 0$.
\end{proof}

\section{The point spectrum of lattice homomorphisms}
\label{sec:lattice-homomorphisms-non-semi-simple}

Equipped with the multidimensional positive homogeneous functional calculus we can now reconsider Example~\ref{exa:counterexample-spectrum} in a more general way. 
Recall that an eigenvalue $\lambda$ of a bounded linear operator $S$ on a complex Banach lattice is called \emph{semi-simple} if $\ker(\lambda-S) = \ker(\lambda-S)^2$ (equivalently, $\ker(\lambda-S) = \ker(\lambda-S)^n$ for each integer $n \ge 1$).

\begin{theorem} \label{thm:spec-ring}
    Let $E$ be a complex Banach lattice and 
    suppose $S:E\to E$ is a Banach lattice homomorphism with an eigenvalue $\lambda \in \bbC$ that is not semi-simple. 
    Then $|\lambda|\bbT\subseteq \pntSpec(S)$.
\end{theorem}

\begin{proof}
    As $\lambda$ is a non-semi-simple eigenvalue there is a $y\neq 0$ such that
        \[ y\in \ker\left((S-\lambda)^2\right)\setminus\ker(S-\lambda). \]
    Then $x\coloneqq (S-\lambda)y \neq 0$ is an eigenvector of $S$ to the eigenvalue $\lambda$, 
    and $S\tmatrix{y\\x} = \tmatrix{\lambda &1\\0 &\lambda}\tmatrix{y\\x}$. 
    Fix an angle $\varphi\in [0,2\pi)$ and define $f\in C_{ph}(\bbC^2,\bbC)$ by
        \[ f(z) = \begin{cases}z_2 \exp(i\varphi\re\frac{\lambda z_1}{z_2}), & z_2\neq 0\\0, & z_2 = 0.\end{cases} \]
    It is easy to check that $f\circ\tmatrix{\lambda &1\\0 &\lambda} = \lambda e^{i\varphi}f$. 
    Therefore Theorem~\ref{thm:funccalc-properties}\ref{thm:funccalc-properties:itm:LatHom} implies
        \[ Sf\tmatrix{y\\x} = f\left(S\tmatrix{y\\x}\right) = f\left(\tmatrix{\lambda &1\\0 &\lambda}\tmatrix{y\\x}\right) = f\circ\tmatrix{\lambda &1\\0 &\lambda}\tmatrix{y\\x} = \lambda e^{i\varphi} f\tmatrix{y\\x}. \]
    Moreover, $f\tmatrix{y\\x}\neq 0$, because $\left|f\tmatrix{y\\x}\right| = |f|\tmatrix{y\\x} = |x| \neq 0$. Thus, $f\tmatrix{y\\x}$ is an eigenvector of $S$ for the eigenvalue $\lambda e^{i\varphi}$.
\end{proof}

\section{The point spectrum of $\overline{T}$}
\label{sec:shape-of-spec-overline-T}

In the following lemma we show that the embedding $\delta_E$ of a complex Banach space $E$ into its free Banach lattice always maps non-zero vectors to non-disjoint vectors.
In the proof of Theorem~\ref{thm:pnt-spec-sun-like} below this will help us to show that the eigenvectors we construct are non-zero (and are thus indeed eigenvectors).
The same result can be shown for real Banach spaces in a similar fashion, but we only consider the complex case to keep the notation more coherent.

\begin{lemma} 
    \label{FBLDisjoinedness}
    Let $E$ be a complex Banach space and $0\neq x_k\in E$ for $k=1,\dots,n$. Then $\modulus{\delta_E(x_1)} \land\dots\land \modulus{\delta_E(x_n)} \neq 0$.
\end{lemma}

\begin{proof}
    For each $k \in \{1, \dots, n\}$ let $X'_k \subseteq X'$ denote the annihilator of $x_k$ in $X'$, i.e.\ 
    \begin{align*}
        X'_k 
        := 
        \big\{ 
            x' \in X' \mid \langle x', x_k \rangle = 0
        \big\}
        .
    \end{align*}
    Then each $X'_k$ is a closed vector subspace of $X'$ that is distinct from $X'$ itself since $x_k \not= 0$, so $\bigcup_{k=1}^n X'_k \not= X'$. 
    Hence, there exists a vector $x' \in X'$ such that $\langle x', x_k \rangle \not= 0$ for each $k$ and thus, $\modulus{\delta_E(x_k)}(x') > 0$
    for all $k$, which implies the claim.
\end{proof}

One can even show a quantitative version of the previous lemma, see Lemma~\ref{FBLDisjoinedness12} below. 
To this end we need the following auxiliary result for general Banach spaces.

\begin{lemma} \label{DisjoinednessfunctionalExists}
    Let $E$ be a Banach space and $0\neq x_1,\dots,x_n\in E$ such that $\norm{x_1} = \dots = \norm{x_n} = 1$. Then there is a functional $x'\in E'$ such that $\norm{x'} \leq 2$ and $\modulus{\langle x',x_k\rangle} \geq \frac{1}{3^{n-1}}$ for all $k=1,\dots,n$.
\end{lemma}

\begin{proof}
    We prove the following slightly stronger result via induction over $n$: 
    there exists a functional $x' \in E'$ such that $\norm{x'} \leq 1+\sum_{k=1}^{n-1}\frac2{3^k}$ and $\modulus{\langle x',x_k\rangle} \geq \frac{1}{3^{n-1}}$ for all $k=1,\dots,n$. 
    
    For $n=1$ this is a consequence of the Hahn--Banach extension theorem. 
    Now assume that we have proved the claim for some $n\in \bbN$ and 
    let $x_{n+1}\in E$ with $\norm{x_{n+1}} = 1$.
    We need to show that there is a $y'\in E'$ such that $\norm{y'} \leq 1+\sum_{k=1}^n\frac2{3^k}$ and  $\modulus{\langle y',x_k\rangle} \geq \frac{1}{3^n}$ for all $k=1,\dots,n+1$. 
    
    If $\modulus{\langle x',x_{n+1}\rangle} \geq \frac{1}{3^n}$ then $x'$ meets those two requirements. 
    So assume that $\modulus{\langle x',x_{n+1}\rangle} < \frac{1}{3^n}$. 
    By the Hahn--Banach extension theorem there is a functional $z'\in E'$ such that $\modulus{\langle z', x_{n+1}\rangle} = 1$ and $\norm{z'} = 1$. 
    We define $y' = x' + \frac2{3^n}z'$. 
    This yields the inequalities
        \[ \modulus{\langle y',x_k\rangle} \geq \frac{1}{3^{n-1}} - \frac2{3^n}\modulus{\langle z',x_k\rangle} \geq \frac1{3^n} \]
    for all $k=1,\dots,n$ and
        \[ \modulus{\langle y',x_{n+1}\rangle} \geq \frac{2}{3^n} - \modulus{\langle x',x_{n+1}\rangle} \geq \frac{1}{3^n}.\]
    Moreover, $\norm{y'} \leq \norm{x'} + \frac{2}{3^n}\norm{z'} \leq 1 + \sum_{k=1}^{n}\frac{2}{3^k}$.
\end{proof}

\begin{lemma} \label{FBLDisjoinedness12}
    Let $E$ be a complex Banach space and $0\neq x_1,\dots,x_n\in E$ such that $\norm{x_1} = \dots = \norm{x_n} = 1$. Then $\norm{\modulus{\delta_E(x_1)}\land\dots\land\modulus{\delta_E(x_n)}} \geq \frac1{2\cdot3^{n-1}}$.
\end{lemma}

\begin{proof}
    We denote $f\coloneqq \modulus{\delta_E(x_1)}\land\dots\land\modulus{\delta_E(x_n)}\in \FBL_\bbC[E]$ to make the notation more digestible. Lemma~\ref{DisjoinednessfunctionalExists} yields a functional $x'\in E'$ such that $\norm{x'} \leq 2$ and $\modulus{\langle x',x_k\rangle} \geq \frac{1}{3^{n-1}}$ for all $k = 1,\dots,n$. This implies
        \[ \norm{f} \geq \norm{f}_\infty \geq f\left(\tfrac12x'\right) \geq \frac{1}{2\cdot3^{n-1}},\]
        where the $\infty$-norm of $f$ is computed over the unit ball of $E'$.
\end{proof}

To motivate the subsequent Theorem~\ref{thm:pnt-spec-sun-like}, 
let us briefly discuss the following question. 
Consider a bounded linear operator $T\colon E \to E$ on a complex Banach space $E$. 
One might ask whether there can be a spectral value $\lambda$ of $\overline{T}$ such that $\modulus{\mu} \neq \modulus{\lambda}$ for all $\mu\in\spec(T)$. 
The following example shows that this can indeed happen.

\begin{example} \label{exa:counterexample-spectrum2}
    Consider the operator $T:\bbC^2\to \bbC^2$ given by the matrix $T=\tmatrix{1&0\\0&2}$. Clearly $\spec(T) = \{1,2\}$. 
    Let $\gamma\in[0,1]$. 
    We show that $2^{1-\gamma}$ is an eigenvalue of $\overline{T}$. 
    Indeed, let $f\in C_{ph}((\bbC^2)',\bbC)$ be given by
        \[ f(z') = \modulus{z_1'}^\gamma\modulus{z_2'}^{1-\gamma} \]
        for all $z' \in \bbC^2 \simeq (\bbC^2)'$.
    Using $\overline{T}f = f\circ T'$ (see \cite[Lemma 3.1 on p.\,27]{OikhbergTaylorTradaceteTroitsky2022}) we calculate
        \[ (\overline{T}f)(z') = f(T'z') = \modulus{z_1'}^\gamma\modulus{2z_2'}^{1-\gamma} = 2^{1-\gamma}f(z'). \]
    Thus, $[1,2]\subseteq \pntSpec(\overline{T})$.
\end{example}

Again we can use the functional calculus to adapt Example~\ref{exa:counterexample-spectrum2} to a more general setting. 
We will do this in the proof of the following theorem.

\begin{theorem}
    \label{thm:pnt-spec-sun-like}
    Let $E$ be a complex Banach space and $T:E\to E$ a bounded linear operator that has a non-zero eigenvalue. 
    Let
    \begin{align*}
        U 
        \coloneqq 
        \left\{\prod_{k=1}^n\frac{\lambda_k}{\modulus{\lambda_k}}
        \middle|\,
        n \in \bbN_0, \; \lambda_1, \dots, \lambda_n\in\pntSpec(T)\setminus\{0\} \right\}
    \end{align*}
    be the subgroup of $\bbT$ generated by the phases of all non-zero eigenvalues of $T$. Then $ru\in \pntSpec(\overline{T})$ for all $u\in U$ and all $r \in (\pntSpri(T),\pntSpr(T))$.
\end{theorem}

\begin{proof}
    Let $u = \prod_{k=1}^n\frac{\lambda_k}{\modulus{\lambda_k}}\in U$ 
    and $\pntSpri(T) < r < \pntSpr(T)$. Thus there are $\lambda_{n+1},\lambda_{n+2}\in\pntSpec(T)$ such that $\modulus{\lambda_{n+1}}< r<\modulus{\lambda_{n+2}}$. We denote the corresponding eigenvectors by $v_k$ and write $x_k\coloneqq \delta_E(v_k)\in \FBL_\bbC[E]$. We write $x\coloneqq (x_1,\dots,x_{n+2}) \in E^{n+2}$ and $D_\lambda\in\bbC^{(n+2)\times (n+2)}$ for the diagonal matrix with the entries $\lambda_1,\dots,\lambda_{n+2}$. 
    Clearly $\overline T x = D_\lambda x$. 
    There are real numbers $e_1,\dots,e_{n+2}\geq 1$ such that $r = \left(\prod_{k=1}^{n+2}\modulus{\lambda_{k}}^{e_k}\right)^{\frac1{m}}$ where $m\coloneqq \sum_{k=1}^{n+2}e_k$. We define $f\in C_{ph}(\bbC^{n+2},\bbC)$ by
    \begin{equation*} \label{eq:usefulFunktion}  
        f(z) 
        = 
        \begin{cases}
            \left( \prod_{k=1}^n\tfrac{z_k}{\modulus{z_k}} \right) \left( \prod_{k=1}^{n+2}\modulus{z_{k}}^{e_k} \right)^{\frac1{m}} & \text{if }z_k \neq 0 \text{ for all }k \\
            0  & \text{else}.
        \end{cases} 
    \end{equation*} 
    Using Theorem~\ref{thm:funccalc-properties}\ref{thm:funccalc-properties:itm:LatHom} we get
    \[ 
        \overline{T}f(x) = f(\overline{T}x) = (f\circ D_\lambda) (x) = ru\,f(x). 
    \]
    From Lemma~\ref{FBLDisjoinedness} we know that $y\coloneqq \modulus{x_1} \land\dots\land \modulus{x_n} \neq 0$, so for $\hat{y}\coloneqq(y,\dots,y)\in E^{n+2}$ Theorem~\ref{thm:funccalc-properties}\ref{thm:funccalc-properties:itm:monotone} implies
    \[
        \modulus{f(x)} = g(\modulus{x}) \geq g(\hat y) = y \neq 0.
    \]
    where $g\in C_{ph}(\bbC^{n+2},\bbC)$ is defined by 
    \begin{equation*} \label{eq:usefulFunktion2}
        g(z) = \left(\prod_{k=1}^{n+2}\modulus{z_k}^{e_k}\right)^{\frac1{m}}. 
    \end{equation*}
    Therefore $ru$ is an eigenvalue of $\overline T$ with the eigenvector $f(x)$.
\end{proof}

\begin{example}
    \label{exa:matrix-sun-like}
    Consider the matrix 
    \begin{align}
        \label{eq:exa:matrix-sun-like:matrix}
        T 
        =
        \begin{pmatrix}
            e^{i\frac{2\pi}{3}} & 0                 \\
            0              & 3 e^{i\frac{2\pi}{4}}
        \end{pmatrix}
    \end{align}
    acting on $\bbC^2$. 
    The point spectrum $\pntSpec(T)$ and the part of $\pntSpec(\overline{T})$ that one obtains from Theorem~\ref{thm:pnt-spec-sun-like} are shown in Figure~\ref{fig:sun}. 
    However, Theorem~\ref{thm:dounut-from-eigenvalues}\ref{thm:dounut-from-eigenvalues:itm:eigen} below will show that $\spec(\overline{T})$ is actually much larger than the sun-like shape in Figure~\ref{fig:sun}.
    \begin{figure}[h!]
      \begin{center}
        \begin{tikzpicture}
          \draw [black, ->] (-2,0) to (2,0) node[anchor=west] {$\bbR$}; % real axis
          \draw [black, ->] (0,-2) to (0,2) node[anchor=south] {$\iu \bbR$}; % imaginary axis
          
          \draw [thick] (0.25,0.433) to (0.75,1.3);
          \draw [thick] (-0.25,0.433) to (-0.75,1.3);
          \draw [thick] (0.25,-0.433) to (0.75,-1.3);
          \draw [thick] (-0.25,-0.433) to (-0.75,-1.3);
          \draw [thick] (0.433,0.25) to (1.3,0.75);
          \draw [thick] (0.433,-0.25) to (1.3,-0.75);
          \draw [thick] (-0.433,0.25) to (-1.3,0.75);
          \draw [thick] (-0.433,-0.25) to (-1.3,-0.75);
          \draw [thick] (0,0.5) to (0,1.5);
          \draw [thick] (0,-0.5) to (0,-1.5);
          \draw [thick] (0.5,0) to (1.5,0);
          \draw [thick] (-0.5,0) to (-1.5,0);
          
          \draw [fill] (-0.25,0.433) circle [radius=0.05];
          \draw [fill] (0,1.5) circle [radius=0.05];
        \end{tikzpicture}
        \caption{
            \label{fig:sun}
            The point spectrum $\pntSpec(T) = \{e^{i\frac{2\pi}{4}}, 3 e^{i\frac{2\pi}{4}}\}$ of the matrix $T$ defined in formula~\eqref{eq:exa:matrix-sun-like:matrix} in Example~\ref{exa:matrix-sun-like} is marked as two bold dots. 
            The part of $\pntSpec(\overline{T})$ known from Theorem~\ref{thm:pnt-spec-sun-like} consists of sun-shaped lines.
        }
      \end{center}
    \end{figure}
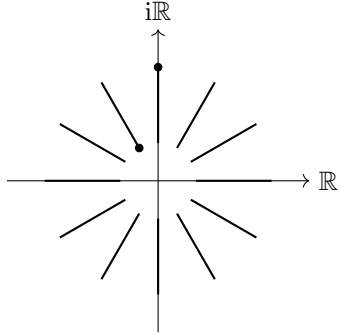
\end{example}

\begin{example} \label{exa:eigenvectors-vanish}
    A question that arises from Theorem~\ref{thm:pnt-spec-sun-like} is if the same technic could be applied to arbitrary lattice homomorphisms to find eigenvalues. However, an important ingredient in the proof of Theorem~\ref{thm:pnt-spec-sun-like} is Lemma~\ref{FBLDisjoinedness}. The non-disjointness of the eigenvectors of $\overline{T}$ ensures that the eigenvectors created using the functional calculus do not vanish. On arbitrary Banach lattices this is not guaranteed.

    Indeed, let $T$ be a lattice homomorphism on $\bbR^n$ with eigenvalues $\lambda_1$ and $\lambda_2$ such that $0<\lambda_1<\lambda_2$ and corresponding eigenvectors $x_1$ and $x_2$. For $\gamma\in[0,1]$ let $f_\gamma\in C_{ph}(\bbC^2,\bbC)$ be defined by 
        \[ f_\gamma(z) = \modulus{z_1}^\gamma\modulus{z_2}^{1-\gamma}. \]
    The vector $f_\gamma(x)$ vanishes if and only if $x_1$ and $x_2$ are disjoint. If $x_1$ and $x_2$ are disjoint, $f_\gamma(x)=0$ follows from the construction of the functional calculus, and if $f_\gamma(x) = 0$ then $x_1$ and $x_2$ must be disjoint by the same argument as in the proof of Theorem~\ref{thm:pnt-spec-sun-like}. 
    As $T$ cannot have infinitely many eigenvalues, it follows that $f_\gamma(x) = 0$ for almost but at most finitely many $\gamma$. 
    But this implies that $x_1$ and $x_2$ are disjoint and therefore $f_\gamma(x) = 0$ even for all $\gamma$.

    A very simple example where a lattice homomorphism has eigenvectors that are not disjoint is the left shift operator $L$ on $\ell^p$ for $p \in [1,\infty]$. 
    Since $\pntSpec(L) = \bbD$ we can choose two eigenvalues $0<\lambda_1<\lambda_2$ with eigenvectors $(x^1_n),(x^2_n)\in\ell^p$ where $x^k_n = \lambda_k^n$ for $k=1,2$. Similar to Theorem~\ref{thm:pnt-spec-sun-like} applying the functional calculus yields that $f(x)$ is an eigenvector for the eigenvalue $\lambda_1^\gamma\lambda_2^{1-\gamma}$.
\end{example}

In Theorem~\ref{thm:pnt-spec-sun-like} the spectral value $0$ does not yield the same information about the spectrum of $\overline{T}$ as the other spectral values of $T$. From the following example we can see that the assertion of Theorem~\ref{thm:pnt-spec-sun-like} does not hold for $0\in \pntSpec(T)$.

\begin{example} \label{exa:projection}
    We choose $T:\bbC^2\to \bbC^2$ with $T=\tmatrix{1 &0\\0&0}$. Clearly $\spec(T) = \pntSpec(T) = \{1,0\}$. 
    A look at Theorem~\ref{thm:pnt-spec-sun-like} might suggest that $(0,1)\subseteq\pntSpec(\overline{T})$, but this is not actually the case:
    as $T$ is a projection, so is $\overline{T}$, and thus $\spec(\overline{T}) = \{0,1\}$.
\end{example}

Finally we can combine Theorem~\ref{thm:spec-ring} and Theorem~\ref{thm:pnt-spec-sun-like} to obtain even more eigenvalues of $\overline{T}$.

\begin{corollary} \label{cor:spec-donut}
    Let $E$ be a complex Banach space and $T:E\to E$ a bounded linear operator with a non semi-simple eigenvalue $0\neq\lambda_0\in \bbC$. Then $r\bbT\subseteq \pntSpec(\overline{T})$ for all $\pntSpri(T) < r < \pntSpr(T)$.
\end{corollary}

\begin{proof}
    As we have seen in Theorem~\ref{thm:spec-ring} there are $v_1,v_2\in E\setminus\{0\}$ such that $Tv_1 = \lambda_0 v_1+v_2$ and $Tv_2 = \lambda_0v_2$. 
    Let $\varphi\in [0,2\pi)$ and $\pntSpri(T) < r < \pntSpr(T)$. Thus there are $\lambda_1,\lambda_2\in\pntSpec(T)$ such that $0<\modulus{\lambda_1}< r<\modulus{\lambda_2}$. We denote the corresponding eigenvectors by $v_3$ and $v_4$ and write $x_k\coloneqq \delta_E(v_k)\in \FBL_\bbC[E]$. 
    We write $x\coloneqq (x_1,\dots,x_4)$ and 
    \[ 
        D_\lambda 
        = 
        \tmatrix{ 
            \lambda_0&1&0&0\\
            0&\lambda_0&0&0\\
            0&0&\lambda_1&0\\
            0&0&0&\lambda_2
        }
        .
    \]
    Then $\overline T x = D_\lambda x$. There are $e_0,e_1,e_2\geq 1$ such that $r = \left(\modulus{\lambda_0}^{e_0}\modulus{\lambda_1}^{e_1}\modulus{\lambda_2}^{e_2}\right)^{\frac1m}$ where $m\coloneqq e_0+e_1+e_2$. We define $f\in C_{ph}(\bbC^4,\bbC)$ by
    \begin{equation*}
        f(z) = 
        \begin{cases} 
            \left(\modulus{z_2}^{e_0}\modulus{z_3}^{e_1}\modulus{z_4}^{e_2}\right)^{\frac1m} \exp(i\varphi\re\frac{\lambda_0 z_1}{z_2}), &\text{if }z_2 \neq 0 \\0, &\text{else}.
        \end{cases} 
    \end{equation*} 
    Then $f\circ D_\lambda = re^{\iu\varphi} f$. 
    Form Theorem~\ref{thm:funccalc-properties}\ref{thm:funccalc-properties:itm:LatHom} we get
    \[ 
        \overline{T}f(x) 
        = 
        f(\overline{T}x) = (f\circ D_\lambda) (x) = re^{\iu\varphi}\,f(x). 
    \]
    From Lemma~\ref{FBLDisjoinedness} we know that $y\coloneqq \modulus{x_2} \land \modulus{x_3} \land \modulus{x_4} \neq 0$ so for $\hat y\coloneqq (y,y,y,y)\in E^4$, Theorem~\ref{thm:funccalc-properties}\ref{thm:funccalc-properties:itm:modulus} and~\ref{thm:funccalc-properties:itm:monotone} implies
        \[ \modulus{f(x)} = g(\modulus{x}) \geq g(\hat y) = y \neq 0 \]
    where $g\in C_{ph}(\bbC^4,\bbC)$ is defined by $g(z) = \left(\modulus{z_2}^{e_0}\modulus{z_3}^{e_1}\modulus{z_4}^{e_2}\right)^{\frac1m}$.
    Therefore $re^{\iu\varphi}$ is an eigenvalue of $\overline T$ to the eigenvector $f(x)$.
\end{proof}

\begin{example}
    \label{exa:matrix-donut}
    Consider the matrix 
    \begin{align}
        \label{eq:exa:matrix-donut:matrix}
        T 
        =
        \begin{pmatrix}   
                    2&1&0&0\\
                    0&2&0&0\\
                    0&0&1&0\\
                    0&0&0&3 \end{pmatrix}
    \end{align}
    acting on $\bbC^4$. Corollary~\ref{cor:spec-donut} yields that $(1,3)\bbT\subseteq \pntSpec(\overline{T})$ and thus $\spec(\overline{T}) = [1,3]\bbT$.
    The point spectrum $\pntSpec(T)$ and the part of $\pntSpec(\overline{T})$ that one obtains from Theorem~\ref{thm:spec-ring}, Theorem~\ref{thm:pnt-spec-sun-like} and Corollary~\ref{cor:spec-donut} are shown in Figure~\ref{fig:donut}. 
    \begin{figure}[h!]
      \begin{center}
        \begin{tikzpicture}
          % spectrum of T^_
          \draw [white!50!gray,fill] (0,0) circle [radius=1.5];
          \draw [white,fill] (0,0) circle [radius=0.5];
          % real axis
          \draw [black, ->] (-2,0) to (2,0) node[anchor=west] {$\bbR$}; 
          % imaginary axis
          \draw [black, ->] (0,-2) to (0,2) node[anchor=south] {$\iu \bbR$}; 
          % theorem 5.1
          \draw [thick] (0.5,0) to (1.5,0);
          % spectrum of T
          \draw [fill] (0.5,0) circle [radius=0.05];
          \draw [fill] (1,0) circle [radius=0.05];
          \draw [fill] (1.5,0) circle [radius=0.05];
          % theorem 6.5
          \draw [thick] (0,0) circle [radius=1];
        \end{tikzpicture}
        \caption{
            \label{fig:donut}
            The point spectrum $\pntSpec(T) = \{1,2,3\}$ of the matrix $T$ defined in~\eqref{eq:exa:matrix-donut:matrix} is marked as three bold dots. 
            The part of $\pntSpec(\overline{T})$ known from Theorem~\ref{thm:spec-ring} and Theorem~\ref{thm:pnt-spec-sun-like} is emphasized as bold lines and $\pntSpec(\overline{T})$ is shown in gray.
        }
      \end{center}
    \end{figure}
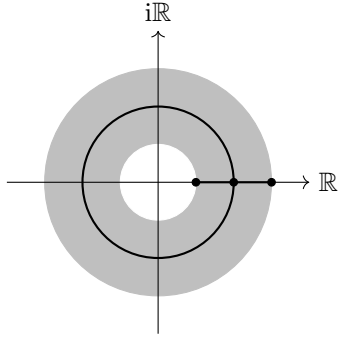
\end{example}

The interval $(\pntSpri(T),\pntSpr(T))$ in the conclusion of Theorem~\ref{thm:pnt-spec-sun-like} does not contain its end points. 
This raises the question if a similar result can be obtained for eigenvalues of $T$ that are all located on the same circle. 
This is indeed possible, as the following theorem shows.

\begin{theorem}
    \label{thm:pnt-spec-circle-subgroup}
    Let $E$ be a complex Banach space and $T:E\to E$ a bounded linear operator. 
    Let $r > 0$ such that $T$ has an eigenvalue of modulus $r$ and let
    \begin{align*}
        U 
        \coloneqq 
        \left\{\prod_{k=1}^n\frac{\lambda_k}{\modulus{\lambda_k}}
        \middle|\,
        n \in \bbN_0, \; \lambda_1, \dots, \lambda_n  \in  \big(\pntSpec(T)\setminus\{0\} \big) \cap r\bbT \right\}
    \end{align*}
    be the subgroup of $\bbT$ generated by the phases of the eigenvalues of $T$ of modulus $r$. 
    Then $r u\in \pntSpec(\overline{T})$ for all $u\in U$.
\end{theorem}

\begin{proof}
    Let $u = \prod_{k=1}^n\frac{\lambda_k}{\modulus{\lambda_k}}\in U$. 
    We may assume that $n \ge 1$ 
    (since $n=0$ implies $u=1$, in which case $ru = r$ is an eigenvalue of $\overline{T}$ since this $\overline{T}$ is a lattice homomorphism and has an eigenvalue of modulus $r$ by assumption). 
    For each $k = 1, \dots, n$ let $v_k \in E$ be an eigenvector of $T$ for the eigenvalue $\lambda_k$ 
    and define $x_k\coloneqq \delta_E(v_k)\in \FBL_\bbC[E]$. 
    Write $D_{\lambda} \in \bbC^{n\times n}$ for the diagonal matrix with the $\lambda_k$ as diagonal entries. 
    Then $\overline{T}x = D_{\lambda}x$ for $x = (x_1, \dots, x_n) \in E^n$. 
    We define $f\in C_{ph}(\bbC^n,\bbC)$ by 
    \[ 
        f(z) =  
        \begin{cases}
            \left( \prod_{k=1}^n\tfrac{z_k}{\modulus{z_k}} \right)  \left( \prod_{k=1}^n \modulus{z_k} \right)^{\frac1{n}} & \text{if } z_k \neq 0 \text{ for all } k, \\
            0  & \text{else}.
        \end{cases} 
    \] 
    Then $f \circ D_\lambda = ru \, f$ and hence, 
    by Theorem~\ref{thm:funccalc-properties}\ref{thm:funccalc-properties:itm:LatHom}, 
    \[
        \overline{T} f(x) 
        = 
        f(\overline{T} x)
        = 
        (f \circ D_\lambda)(x)
        = 
        ru f(x)
        .
    \]
    Observe that, by Theorem~\ref{thm:funccalc-properties}\ref{thm:funccalc-properties:itm:positive}, $\modulus{f(x)} \ge \modulus{x_1} \land \dots \land \modulus{x_n}$, which is non-zero according to Lemma~\ref{FBLDisjoinedness}, so $ru$ is indeed an eigenvalue of $\overline{T}$.
\end{proof}

\section{The spectrum of $\overline{T}$ for essential roots of unity}
\label{sec:spec-overline-ess-root}

The tools developed so far enable us to characterize $\spec(\overline{T})$ if $T$ is essentially a root of unity (Theorem~\ref{thm:roots-of-unity-general}). 
First, we need the following two auxiliary results. 

\begin{lemma}
    \label{lem:power-not-in-spec}
    Let $E$ be a complex Banach lattice, let $S: E \to E$ be a lattice homomorphism, and consider a complex number $\lambda \not\in \spec(S)$. 
    For each integer $n_0 \ge 1$ there exists an $n \ge n_0$ such that $\lambda^n \not\in \spec(S^n)$.
\end{lemma}

\begin{proof}
    We may assume that $\lambda \not= 0$.
    Set $\mu \coloneqq \frac{\lambda}{\modulus{\lambda}} \in \bbT$, i.e.\ $\lambda = \mu \modulus{\lambda}$. 
    Choose $n = n_0$ if $\mu$ is not a root of unity. 
    If $\mu$ is a root of unity of order $m \ge 1$, choose $n \ge n_0$ to be an integer that is co-prime to $m$. 

    Assume now for a contradiction that $\lambda^n \in \spec(S^n)$. 
    By the spectral mapping theorem for polynomials there exists an $n$th root of unity $\xi$ such that $\xi \lambda \in \spec(S)$. 
    Since $S$ is a lattice homomorphism, its spectrum is cyclic.
    If $\xi \mu$ is not a root of unity, then it follows that the whole circle of radius $\modulus{\lambda}$ is contained in $\spec(S)$, which contradicts $\lambda \not\in \spec(S)$. 
    So $\xi \mu$ is a root of unity, and hence so is $\mu$. 
    By the choice of $n$, the order $m$ of $\mu$ is coprime to $n$, so there exists an integer $k \ge 1$ such that $nk \equiv 1$ (mod $m$).
    Hence, 
    \begin{align*}
        \spec(S) 
        \ni 
        (\xi \mu)^{nk} \lambda 
        = 
        \lambda
        ,
    \end{align*}
    which is again a contradiction.
\end{proof}

\begin{lemma}
    \label{lem:zero-pole-shift}
    Let $E$ be a complex Banach space and let $T: E \to E$ be a bounded linear operator. 
    Let $0 \in \spec(T)$ be a pole of the resolvent, 
    with associated spectral projection $Q$ and set $P \coloneqq \id_E - Q$.
    If $\lambda_0 \in \bbC$ and $S \coloneqq PT + \lambda_0 Q$, 
    then $\spec(\overline{T}) \setminus \{0\} \subseteq \spec(\overline{S})$. 
\end{lemma}

\begin{proof}
    Let $\lambda \in \bbC \setminus \{0\}$. 
    We assume that $\lambda \not\in \spec(\overline{S})$ and we need to show that $\lambda \not\in \spec(\overline{T})$.

    Since $0$ is a pole of the resolvent, the restriction of $T$ to the range of $Q$ is nilpotent, i.e.\ there exists an integer $n \ge 1$ such that $T^n Q = 0$. 
    It follows from Lemma~\ref{lem:power-not-in-spec} that, by making $n$ larger if necessary, we can achieve that $\lambda^n \not\in \spec(\overline{S}^n)$.

    Observe that $T^n = T^nP = PT^n = S^nP = PS^n$, and hence the same equality holds for the  lifted operators $\overline{T}$, $\overline{S}$, $\overline{P}$. 
    Therefore,
    \begin{align*}
        (\lambda^n - \overline{T}^n) 
        \Big( 
            \overline{P}(\lambda^n-\overline{S}^n)^{-1} + \tfrac{1}{\lambda^n}(\overline{\id_E} - \overline{P})
        \Big) 
        & = 
        \overline{P} + (\overline{\id_E} - \overline{P}) - \frac{1}{\lambda^n}\overline{T}^n (\overline{\id_E} - \overline{P}) 
        \\ 
        & = 
        \overline{\id_E} - {\lambda^n} \overline{T}^n (\overline{\id_E} - \overline{P}) 
        = 
        \overline{\id_E}
    \end{align*}
    since $\overline{T}^n (\overline{\id_E} - \overline{P}) = \overline{T}^n - \overline{T}^n \overline{P} = 0$. 
    Since the factors on the left hand side of the above equation commute, we conclude that $\lambda^n - \overline{T}^n$ is invertible, so $\lambda^n \not\in\spec(\overline{T}^n)$. 
    The spectral mapping theorem for polynomials thus implies that $\lambda \not\in \spec(\overline{T})$.
\end{proof}

\begin{theorem} \label{thm:roots-of-unity-general}
    Let $E$ be a complex Banach space and let $T: E \to E$ be a bounded linear operator with $\spec(T) \not= \{0\}$. 
    \begin{enumerate}[label=\upshape(\alph*)]
        \item\label{thm:roots-of-unity-general:itm:id} 
        If $T$ is a root of unity, then $\spec(\overline{T})$ is the subgroup of $\bbT$ generated by $\spec(T)$.
        
        \item\label{thm:roots-of-unity-general:itm:proj} 
        If $T$ is a proper essential root of unity, then $\spec(\overline{T})$ consists of $0$ and of the subgroup of $\bbT$ generated by $\spec(T) \setminus \{0\}$.
    \end{enumerate}
\end{theorem}

\begin{proof}
    In each case, $G \subseteq \bbT$ denote the subgroup generated by $\spec(T) \setminus \{0\} \subseteq \bbT$.
    
    \begin{description}[leftmargin=0.65cm]
        \item[\ref{thm:roots-of-unity-general:itm:id}]
        Obviously $\spec(T) \subseteq G$. 
        On the other hand, $\spec(T) = \pntSpec(T)$ by Lemma~\ref{lem:spectral-decomposition}\ref{lem:spectral-decomposition:itm:id}, so $G \subseteq \pntSpec(\overline{T}) \subseteq \spec(\overline{T})$ according to Theorem~\ref{thm:pnt-spec-circle-subgroup}. 

        \item[\ref{thm:roots-of-unity-general:itm:proj}] 
        One has $\spec(T) = \pntSpec(T)$ by Lemma~\ref{lem:spectral-decomposition}\ref{lem:spectral-decomposition:itm:proj}, so $G \subseteq \pntSpec(\overline{T}) \subseteq \spec(\overline{T})$ according to Theorem~\ref{thm:pnt-spec-circle-subgroup}. 
        Moreover, $0 \in \pntSpec(T)$, so also $0 \in \pntSpec(\overline{T}) \subseteq \spec(\overline{T})$.
        
        To prove the converse inclusion choose $\lambda_0 \in \spec(T) \setminus \{0\}$; 
        such a number exists since $\spr(T) > 0$.
        Note that $0$ is a pole of the resolvent of $T$ according to Lemma~\ref{lem:spectral-decomposition}\ref{lem:spectral-decomposition:itm:proj}; let $Q$ denote the associated spectral projection, $P \coloneqq \id_E - Q$ and $S \coloneqq PT + \lambda_0 Q$. 
        Lemma~\ref{lem:zero-pole-shift} shows that $\spec(\overline{T}) \setminus \{0\} \subseteq \spec(\overline{S})$. 
        Moreover, $S$ is a root of unity and $\spec(S) = \spec(T) \setminus \{0\}$. 
        Hence, it follows from~\ref{thm:roots-of-unity-general:itm:id} that $\spec(\overline{S}) = G$.
        \qedhere
    \end{description}
\end{proof}

\section{The approximate point spectrum of lattice homomorphism}
\label{sec:app-point-spec-overline-T}

The results presented this far might lead one to the following impression: 
if $T$ is a matrix with only positive semi-simple eigenvalues, then $\spec(\overline{T}) \subseteq [0,\infty)$; 
however, this is not actually true. 
We first give a counterexample and then generalize the underlying idea to a result about the approximate point spectrum of lattice homomorphisms.

\begin{example} \label{exa:pos-real}
    We again turn to the matrix $T=\tmatrix{1&0\\0&2}$ from Example~\ref{exa:counterexample-spectrum2}. 
    There is an approximate eigenvalue $(f_n)\subseteq C_{ph}((\bbC^2)',\bbC)$ for $-1$ given by
        \[ f_n(z') = \begin{cases}
                        \sqrt[n]{\frac{\modulus{z_2'}}{\modulus{z_1}}}\modulus{z_1'}\sin\left(\frac{\pi}{\ln(2)}\ln\left(\frac{\modulus{z_2'}}{\modulus{z_1'}}\right)\right), & \text{if } \modulus{z_1'}>0 \\
                        0, & \text{if } \modulus{z_1'} = 0.
                    \end{cases} \]
    Indeed, as in Example~\ref{exa:counterexample-spectrum2} we can use the formula $\overline{T}f_n = f_n\circ T'$ to calculate
        \[ (\overline{T}f_n+f_n)(z') = f(T'z')+f(z') = (1-\sqrt[n]{2})f_n(z') \]
    for all $z' \in \bbC^2 \simeq (\bbC^2)'$;
    which converges uniformly to $0$ as $n \to \infty$. 
    Generalizing this idea leads to Theorem~\ref{thm:dounut-from-eigenvalues} below, which can be used together with Lemma~\ref{lem:spectral-radius} to obtain $\appSpec(T) = \spec(T) = [1,2]\bbT$.
\end{example}

The idea behind Example~\ref{exa:pos-real} can be turned into a general theorem, which will be very useful in our general analysis of the spectrum of $\overline{T}$ in the next section.

\begin{theorem} \label{thm:dounut-from-eigenvalues}
    Let $E$ be a complex Banach lattice and $S:E\to E$ a lattice homomorphism with the  spectral values $\lambda_1,\lambda_2\in \bbC$ such that $0<\modulus{\lambda_1}<\modulus{\lambda_2}$.
    \begin{enumerate}[label=\upshape(\alph*)]
        \item\label{thm:dounut-from-eigenvalues:itm:eigen} 
        Let $\lambda_1, \lambda_2$ be eigenvalues of $S$ with corresponding eigenvectors $x_1,x_2 \in E$. 
        If $x_1$ and $x_2$ are not disjoint, then $[\modulus{\lambda_1},\modulus{\lambda_2}]\bbT \subseteq \appSpec(S)$. 

        \item\label{thm:dounut-from-eigenvalues:itm:approx} 
        More generally, let $\lambda_1, \lambda_2$ be approximate eigenvalues of $S$ with corresponding approximate eigenvectors $(x_n),(y_n)$. 
        Assume that $(x_n)$ and $(y_n)$ are not approximately disjoint, i.e.\ that $(\modulus{x_n}\land \modulus{y_n}) \notin c_0(E)$. 
        Then $[\modulus{\lambda_1},\modulus{\lambda_2}]\bbT \subseteq \appSpec(S)$. 
    \end{enumerate}
\end{theorem}

\begin{proof}
    \begin{description}[leftmargin=0.65cm]
        \item[\ref{thm:dounut-from-eigenvalues:itm:eigen}]
        Let $\varphi\in [0,2\pi)$ and $r\in [\modulus{\lambda_1},\modulus{\lambda_2}]$. 
        We denote the eigenvectors corresponding to $\lambda_1$ and $\lambda_2$ by $x_1$ and $x_2$ and set $x\coloneqq (x_1,x_2) \in E^2$. 
        Let $D_\lambda\in\bbC^{2\times 2}$ denote the diagonal matrix with the entries $\lambda_1$ and $\lambda_2$. 
        Clearly $S x = D_\lambda x$. 
        There are $e_1,e_2\geq 1$ such that $r = (\modulus{\lambda_1}^{e_1}\modulus{\lambda_1}^{e_2})^{\frac1{e_1+e_2}}$. 
        For each $n>\tfrac{e_1+e_2}{e_1}$ the function $f_n: \bbC^2 \to \bbC$ defined by
        \begin{equation*} \label{usefulFunktion3}  
            f_n(z)
            = 
            \begin{cases}
                (\modulus{z_1}^{e_1}\modulus{z_2}^{e_2})^\frac{1}{e_1+e_2}
                \sqrt[n]{\frac{\modulus{z_2}}{\modulus{z_1}}}
                \exp\left(\frac{\iu\varphi}{\ln\left(\frac{\modulus{\lambda_2}}{\modulus{\lambda_1}}\right)}\ln\left(\frac{\modulus{z_2}}{\modulus{z_1}}\right)\right),
                & \text{if } \modulus{z_1}>0 
                \\
                0, & \text{if } \modulus{z_1}=0
            \end{cases} 
        \end{equation*} 
        is continuous and is thus an element of $C_{ph}(\bbC^2,\bbC)$.
        Using Theorem~\ref{thm:funccalc-properties} one can check that
        \[ Sf_n(x) - re^{\iu\varphi}f_n(x) = f_n\circ D_\lambda (x)-re^{\iu\varphi}f_n(x) = g_n(x). \]
        where $g_n\in C_{ph}(\bbC^2,\bbC)$ is defined as
            \[ g_n(z)
            =
            \begin{cases}
                \Big(\sqrt[n]{{\frac{\modulus{\lambda_2}}{\modulus{\lambda_1}}}}-1\Big)re^{\iu\varphi}f_n(z),
                & \text{if } \modulus{z_1} > 0 
                \\
                0,
                & \text{if } \modulus{z_1} = 0.
            \end{cases}\]
        By Theorem~\ref{thm:funccalc-properties}\ref{thm:funccalc-properties:itm:bounded} $\norm{g_n(x)}$ tends to $0$ for $n\to \infty$, because $\norm{g_n}$ does so.
        Finally we show that $f_n(x)$ does not converge to $0$. 
        By assumption we know that $y\coloneqq \modulus{x_1} \land \modulus{x_2} \neq 0$, 
        so Theorem~\ref{thm:funccalc-properties}\ref{thm:funccalc-properties:itm:monotone} implies
            \[ \modulus{f_n(x)}  = \modulus{f_n(\modulus{x})} \geq \modulus{f_n\tmatrix{y\\ y}} = \modulus{y}; \] 
        here we used that each function $f_n$ is increasing on $\bbC^2_+ = \bbR^2_+$ since $n>\tfrac{e_1+e_2}{e_1}$.  
        So $\norm{f_n(x)} \geq \norm{y}$ for all $n$.

        \item[\ref{thm:dounut-from-eigenvalues:itm:approx}]
        We apply \cite[Theorem 4.1.6 on p.\,252]{MeyerNieberg1991} to $T$. 
        By choosing the Fréchet filter in this reference we get that 
        \[ 
            \hat T:\ell^\infty(E)/c_0(E) \to \ell^\infty(E)/c_0(E),
            \quad 
            \hat T (v_n + c_0(E)) = (Tv_n) + c_0(E) 
        \] 
        is a lattice homomorphism satisfying $\appSpec(\hat T) = \pntSpec(\hat T) = \appSpec(T)$. 
        By assumption $(x_n),(y_n)$ are not lattice disjoint in $\ell^\infty(E)/c_0(E)$, so we can apply~\ref{thm:dounut-from-eigenvalues:itm:eigen} to $\hat T$ and thus obtain $[\modulus{\lambda_1},\modulus{\lambda_2}]\bbT\subseteq \appSpec(\hat T) = \appSpec(T)$.
        \qedhere 
    \end{description}
\end{proof}

Let us briefly explain why Theorem~\ref{thm:dounut-from-eigenvalues} can be applied to liftings of operators to free Banach lattices:

\begin{remark}
    \label{rem:approx-applicable}
    If $T:E\to E$ is a bounded linear operator on a Banach space $E$ with eigenvalues $0<\modulus{\lambda_1} < \modulus{\lambda_2}$, Theorem~\ref{thm:dounut-from-eigenvalues}\ref{thm:dounut-from-eigenvalues:itm:eigen} yields that $[\modulus{\lambda_1},\modulus{\lambda_2}]\bbT \subseteq \appSpec(\overline{T})$ since the corresponding eigenvectors of $\overline{T}$ are never lattice disjoint by Lemma~\ref{FBLDisjoinedness}. 
    
    Similarly if $\lambda_1,\lambda_2$ are approximate eigenvalues of $T$, Lemma~\ref{FBLDisjoinedness12} implies that the corresponding approximate eigenvectors are not approximately disjoined. Thus in this case Theorem~\ref{thm:dounut-from-eigenvalues}\ref{thm:dounut-from-eigenvalues:itm:approx} yields $[|\lambda_1|,\modulus{\lambda_2}]\bbT\subseteq\appSpec(\overline{T})$.
\end{remark}

It is illuminating to discuss the proof of Theorem~\ref{thm:dounut-from-eigenvalues}\ref{thm:dounut-from-eigenvalues:itm:eigen} in the following concrete example.

\begin{example}
    Let $p \in [1,\infty]$.
    We revisit the left shift operator $L$ on $\ell^p$ from Example~\ref{exa:eigenvectors-vanish}. There we saw that for eigenvalues $0<\lambda_1<\lambda_2<1$ the eigenvectors are not disjoint, so Theorem~\ref{thm:dounut-from-eigenvalues}\ref{thm:dounut-from-eigenvalues:itm:eigen} yields $\bbD\subseteq\appSpec(L)$. 
    Of course we know that $\bbD$ is even in the point spectrum of $L$; 
    this motivates the question if one can sometimes obtain eigenvalues rather than just approximate eigenvalues from the proof technique of Theorem~\ref{thm:dounut-from-eigenvalues}\ref{thm:dounut-from-eigenvalues:itm:eigen}. 

    In the present example this is possible: 
    the sequence $(f_n)$ in the proof converges pointwise to a positively homogeneous, but non-continuous function $f \colon \bbC^2 \to \bbC$. 
    On $\ell^p$, one can define a functional calculus for non-continuous, positively homogeneous functions componentwise. 
    So if $x_1, x_2 \in \ell^p$ are eigenvectors for the eigenvalues $\lambda_1, \lambda_2$ of $L$, then $f(x_1, x_2)$ is an element of $\ell^p$, and it is indeed an eigenvector of $L$ for the eigenvalue $r e^{i\varphi}$.
\end{example}

\section{A characterization of the spectrum of $\overline{T}$}
\label{sec:spec-char}

In this section we finally prove our main result, a full characterization of $\spec(\overline{T})$ in terms of the spectral properties of $T$ (Theorem~\ref{thm:spec-char-general}).
Before we can do so, we still need to analyze in the case $0 \in \spec(T)$ and $0 < \spri(T)$ whether the numbers $\lambda\in\bbC$ with $0 < \modulus{\lambda} < \spri(T)$ are in the spectrum of $\overline{T}$. 
We do this in Lemmas~\ref{lem:dounut-empty-general} and~\ref{lem:zero-essential-singularity} and in Corollary~\ref{cor:zero-essential-singularity}; 
it turns out that the answer depends on whether $0$ is a pole or an essential singularity of the resolvent of $T$.

\begin{lemma} \label{lem:dounut-empty-general}
    Let $T$ be a bounded linear operator on a Banach space $E$ such that $\{0\}\subsetneq \spec(T)$ and $0$ is a pole of the resolvent of $T$. Then $\spec(\overline{T})\subseteq [q(T),r(T)]\bbT\cup\{0\}$.
\end{lemma}

\begin{proof}
    Let $Q$ be the spectral projection of $T$ for the spectral value $0$ and set $P \coloneqq \id_E-Q$. 
    We define $S=PT+\lambda_0Q$ for some $\lambda_0\in\spec(T)\setminus\{0\}$. 
    Then $\spec(S) = \spec(T) \setminus\{0\}$, 
    so it follows from Lemma~\ref{lem:spectral-radius} (applied to the operator $S$) 
    that $\spec(\overline{S}) \subseteq [\spri(S), \spr(S)] \bbT = [\spri(T), \spr(T)] \bbT$. 
    Moreover, Lemma~\ref{lem:zero-pole-shift} shows that $\spec(\overline{T}) \setminus \{0\} \subseteq \spec(\overline{S})$. 
\end{proof}

\begin{lemma} \label{lem:zero-essential-singularity}
    Let $E$ be a complex Banach lattice and $S:E\to E$ a lattice homomorphism such that $\{0\}\subsetneq\spec(S)$. Let $(w_n)\subseteq E$ be an approximate eigenvector corresponding to $\spr(S)$ and assume there are sequence $(\alpha_n)\subseteq (0,\infty)$, $(\lambda_n)\subseteq \bbC \setminus \{0\}$, and bounded sequences $(x_n),(y_n)\subseteq E$ such that $Sy_n = -\tfrac{x_n}{\alpha_n}+\lambda_ny_n$, $\lambda_n \to 0$ and $\alpha_n\modulus{\lambda_n}^k\to\infty$ as $n\to\infty$ for every $k\in \bbN$. If $(w_n)$ and $(y_n)$ are not approximately disjoint, i.e.\ if $\modulus{w_n} \land \modulus{y_n} \not\to 0$, then $\spec(S)= \spr(S) \overline{\bbD}$.
\end{lemma}

\begin{proof}
    By means of rescaling $S$ we can assume $\spr(S) = 1$. 
    Moreover, by replacing all sequences with subsequences may assume that $\liminf_n \norm{\modulus{w_n} \land \modulus{y_n}} > 0$.
    
    Fix $r\in(0,1)$.
    By dropping finitely many elements of all sequences, 
    we can achieve that that $\modulus{\lambda_n} \le r$ for all indices $n$. 
    For each $n$, choose $\gamma_n \in (0,1]$ such that $r = \modulus{\lambda_n}^{\gamma_n}$, 
    i.e.\ $\gamma_n \coloneqq \tfrac{\ln r}{\ln \modulus{\lambda_n}}$. 
    We define $h_n,\tilde h_n\in C_{ph}(\bbC^2,\bbC)$ by
    \[
        h_n(z) \coloneqq \modulus{z_1}^{1-\gamma_n}\modulus{z_2}^{\gamma_n} \quad\text{and}\quad 
        g_n(z) \coloneqq \modulus{z_1}^{1-\gamma_n}\modulus{\tfrac{z_2}{\alpha_n}}^{\gamma_n}
    \]
    We will first show that $\left(h_n\tmatrix{w_n\\y_n}\right)\subseteq E$ is an approximate eigenvector of $S$ for $r$, so $r$ is an approximate eigenvalue of $S$. 
    To this end we calculate for all $z,\tilde z\in \bbC^2$ and all $n$
    \begin{align*}
        \modulus{h_n(z) - rh_n(\tilde z)} 
        &= \modulus{\modulus{z_1}^{1-\gamma_n}\modulus{z_2}^{\gamma_n} - \modulus{\tilde z_1}^{1-\gamma_n}\modulus{\lambda_n\tilde z_2}^{\gamma_n}} \\
        &\leq \modulus{z_1}^{1-\gamma_n} \modulus{\modulus{z_2}^{\gamma_n} - \vphantom{\Big|} \modulus{\lambda_n\tilde z_2}^{\gamma_n}} + \modulus{\modulus{z_1}^{1-\gamma_n}- \modulus{\tilde z_1}^{1-\gamma_n}}\modulus{\lambda_n\tilde z_2}^{\gamma_n}.
    \end{align*}
    For all $s \in (0,1]$ and all $a,b \ge 0$ one has $\modulus{b^s-a^s} \le \modulus{b-a}^s$, so it follows that
    \begin{align*} 
        \modulus{h_n(z) - rh_n(\tilde z)} 
        &\leq \modulus{z_1}^{1-\gamma_n} \modulus{z_2-\lambda_n\tilde z_2}^{\gamma_n} + \modulus{z_1-\tilde z_1}^{1-\gamma_n}\modulus{\lambda_n\tilde z_2}^{\gamma_n} \\
        &= h_n\tmatrix{z_1\\z_2-\lambda_n\tilde z_2} + h_n\tmatrix{z_1-\tilde z_1\\\lambda_n\tilde z_2}.
    \end{align*}
    Thus, Theorem~\ref{thm:funccalc-properties}\ref{thm:funccalc-properties:itm:positive} implies
    \begin{align*} 
        \modulus{Sh_n\tmatrix{w_n\\y_n} - r h_n\tmatrix{w_n\\y_n}} 
        &\leq h_n\tmatrix{Sw_n\\Sy_n-\lambda_ny_n} + h_n\tmatrix{Sw_n-w_n\\\lambda_ny_n} \\
        &= g_n\tmatrix{Sw_n\\-x_n} + h_n\tmatrix{Sw_n-w_n\\\lambda_ny_n},
    \end{align*}
    where the equality uses the assumption $Sy_n = -\tfrac{x_n}{\alpha_n}+\lambda_ny_n$.
    Applying Theorem~\ref{thm:funccalc-properties}\ref{thm:funccalc-properties:itm:bounded} yields
    \begin{align*} 
        & \norm{Sh_n\tmatrix{w_n\\y_n} - r h_n\tmatrix{w_n\\y_n}} \\ 
        & \quad \leq \norm{g_n}(\norm{Sw_n} + \norm{x_n}) 
         + \underbrace{\norm{h_n}}_{\le 1}(\underbrace{\norm{Sw_n-w_n}}_{\to 0} + \underbrace{\modulus{\lambda_n}}_{\to 0} \norm{y_n}).
    \end{align*}
    Let us show that also $\norm{g_n} \to 0$.
    Indeed, fix an integer $k \ge 1$ and recall that $\alpha_n\modulus{\lambda_n}^k\to\infty$ as $n \to \infty$ by assumption. 
    For large $n$ one thus has $\alpha_n \ge \tfrac1{\modulus{\lambda_n}^k}$ and hence
    \[
        \ln \alpha_n^{-\gamma_n} 
        = 
        -\gamma_n\ln\alpha_n 
        \le 
        k \gamma_n \ln \modulus{\lambda_n}
        =
        k \ln r
        .
    \]
    So $\ln \alpha_n^{-\gamma_n} \to -\infty$ and therefore, $\alpha_n^{-\gamma_n} \to 0$ as $n \to \infty$. 
    This implies that $\norm{g_n} \to 0$, as claimed.

    We have thus proved that $Sh_n\tmatrix{w_n\\y_n} - r h_n\tmatrix{w_n\\y_n} \to 0$. 
    For $v_n\coloneqq\modulus{w_n}\land \modulus{y_n}$,
    \[
        \modulus{h_n\tmatrix{w_n\\y_n}} = h_n\tmatrix{\modulus{w_n}\\\modulus{y_n}} \geq \modulus{h_n\tmatrix{v_n\\v_n}} = v_n 
    \]
    by Theorem~\ref{thm:funccalc-properties}\ref{thm:funccalc-properties:itm:monotone}, 
    and $\liminf_n \norm{v_n} > 0$ as ensured at the beginning of the proof. 
    Hence, $\left(h_n\tmatrix{w_n\\y_n}\right)$ is indeed an approximate eigenvector of $S$ with approximate eigenvalue $r$. 
    
    Finally, observe that the approximate eigenvectors $\left(h_n\tmatrix{w_n\\y_n}\right)$ for $r$ and $(w_n)$ for $1$ both dominate the sequence $(v_n)$. 
    Hence, Theorem~\ref{thm:dounut-from-eigenvalues}\ref{thm:dounut-from-eigenvalues:itm:approx} is applicable and implies that $(r,1) \, \bbT \subseteq \appSpec(S) \subseteq \spec(S)$. 
    Since $r\in (0,1)$ was arbitrary and $\spec(S)$ is closed, we conclude that $\overline{\bbD} \subseteq \spec(S)$, which proves the claim.
\end{proof}

\begin{corollary} \label{cor:zero-essential-singularity}
    Let $T$ be a bounded linear operator on a complex Banach space $E$. If $0$ is an isolated spectral value of $T$ and an essential singularity of the resolvent, then $\spec(\overline{T}) = \spr(T)\overline{\bbD}$. 
\end{corollary}

\begin{proof}
    If $\spr(T) = 0$ there is nothing to show, so let us assume $\spr(T)>0$ and let $(\tilde w_n)\subseteq E$ be an approximate eigenvector for $\spr(T)$. 
    
    Since $0$ is an essential singularity of the resolvent, it follows for each integer $n \ge 1$ that $\lambda^n (\lambda \id_E - T)^{-1}$ is unbounded for $\lambda$ in any punctured neighbourhood of $0$. 
    Hence, for each $n$ we can find a number $\lambda_n$ in the resolvent set of $T$ such that $\modulus{\lambda_n} \le \frac{1}{n}$ and $\modulus{\lambda_n}^n \norm{(\lambda_n \id_E - T)^{-1}} \ge n$.
    For each $n$ there exists a vector $\tilde x_n \in E$ of norm $1$ such that $\alpha_n \coloneqq \norm{(\lambda_n \id_E - T)^{-1} \tilde x_n} \ge \frac{1}{2} \norm{(\lambda_n \id_E - T)^{-1}}$. 

    We have $\lambda_n \to 0$, and for all integers $n \ge k \ge 1$ one has $\alpha_n \modulus{\lambda_n}^k \ge \modulus{\lambda_n}^n \alpha_n \ge \frac{n}{2} \ge \frac{k}{2}$, so $\alpha_n \modulus{\lambda_n}^k \to \infty$ as $n \to \infty$.

    We define the normalized vector $\tilde y_n \coloneqq \tfrac1{\alpha_n} (\lambda_n \id_E - T)^{-1} \tilde x_n \in E$ and thus obtain 
    $T\tilde y_n = -\tfrac{\tilde x_n}{\alpha_n} + \lambda_n\tilde y_n$ for each $n$. 
    Finally we set $w_n \coloneqq \delta_E\tilde w_n$, $y_n \coloneqq \delta_E\tilde y_n$ and $x_n \coloneqq \delta_E\tilde x_n$ for all $n$. 
    It follows from Lemma~\ref{FBLDisjoinedness12} that $\norm{\modulus{w_n} \land \modulus{y_n}} \ge \frac{1}{6}$ for all $n$. 
    Thus, Lemma~\ref{lem:zero-essential-singularity} is applicable and gives $\sigma(\overline{T}) = \spr(\overline{T}) \overline{\bbD} = \spr(T) \overline{\bbD}$.
\end{proof}

Finally, we can combine our results so far to obtain the following Theorem~\ref{thm:spec-char-general}, which is the main result of the paper.
As mentioned in the introduction if $\spec(T) = \{0\}$ we know $\spec(\overline{T}) = \{0\}$ since $\spr(T) = \spr(\overline{T})$, so we only consider the case $\spr(T)>0$ and since one can simply rescale $T$ we assume $\spr(T) = 1$.

\begin{theorem} \label{thm:spec-char-general}
    Let $E$ be a complex Banach space and  
    let $T \colon E \to E$ be a bounded linear operator with spectral radius $\spr(T) = 1$. 
    The spectrum of $\overline{T}$ can be described by the following cases. 
    \begin{center}
        \begin{adjustbox}{max width=\textwidth}
            \begin{tabu}{|c||c|c|}
                \hline
                & 
                \parbox[c][1.5cm]{0.38\textwidth}{$T$ is essentially a root of unity} & 
                \parbox[c][1.5cm]{0.38\textwidth}{$T$ is not essentially a root of unity} \\ 
                %\tabucline[1.5pt]{1-4}
                \hhline{|=||=|=|}
                \parbox[c][1.5cm]{0.14\textwidth}{$0\not\in\spec(T)$} & 
                \parbox[c][1.5cm]{0.38\textwidth}{$\spec(\overline{T})$ is the multiplicative subgroup of $\bbT$ generated by $\spec(T)$} &
                $\spec(\overline{T}) = [\spri(T),1]\,\bbT$ \\ 
                \cline{1-3}
                \multirow{2}{*}{\parbox[c][1.5cm]{0.14\textwidth}{$0\in\spec(T)$}} & 
                \multirow{2}{*}{\parbox[c][3cm]{0.38\textwidth}{$\spec(\overline{T})$ consists of $0$ and the multiplicative subgroup of $\bbT$ generated by $\spec(T)\setminus\{0\}$ }} & 
                \parbox[c][1.5cm]{0.38\textwidth}{if $0$ is pole of the resolvent of $T$: $\spec(\overline{T}) = [\spri(T),1]\,\bbT\cup\{0\}$} \\
                \cline{3-3}
                & & 
                \parbox[c][1.5cm]{0.38\textwidth}{if $0$ is not a pole of the resolvent of $T$: $\spec(\overline{T}) = \overline{\bbD}$} \\
                \hline
            \end{tabu}
        \end{adjustbox}
    \end{center}
\end{theorem}

\begin{proof}
    We consider all five cases in the table separately. 
    \begin{description}[leftmargin=0.65cm]
        \item[Case~1:] 
        \emph{$0\not\in\spec(T)$ and $T$ is (essentially) a root of unity.}
        This is Theorem~\ref{thm:roots-of-unity-general}\ref{thm:roots-of-unity-general:itm:id}.

        \item[Case~2:] 
        \emph{$0\in\spec(T)$ and $T$ is essentially a root of unity.}
        This is Theorem~\ref{thm:roots-of-unity-general}\ref{thm:roots-of-unity-general:itm:proj}.
        
        \item[Case~3:] 
        \emph{$0\not\in\spec(T)$ and $T$ is not essentially a root of unity.}
        
        If $\spri(T) = \spr(T)$ (in other words, if $\spec(T)\subseteq \bbT$) Proposition~\ref{prop:hom-root-full-circle}\ref{prop:hom-root-full-circle:itm:id} implies $\spec(\overline{T}) = \bbT$ since $\overline{T}$ is not essentially a root of unity. 
        If $\spri(T)<\spr(T)$ the claim follows from Lemma~\ref{FBLDisjoinedness12} (for $n=2$) and Theorem~\ref{thm:dounut-from-eigenvalues}\ref{thm:dounut-from-eigenvalues:itm:approx} since every point in $\partial \spec(T)$ is an approximate eigenvalue of $T$.

        \item[Case~4:] 
        \emph{$0\in\spec(T)$, $T$ is not essentially a root of unity and $0$ is a pole of the resolvent.}

        From Lemma~\ref{lem:dounut-empty-general} we get $\spec(\overline{T})\subseteq [\spri(T),\spr(T)]\bbT\cup\{0\}$ right away, so we only have to show the converse inclusion.

        If $\spec(T) \setminus \{0\} \not\subseteq \bbT$, 
        then Theorem~\ref{thm:dounut-from-eigenvalues}\ref{thm:dounut-from-eigenvalues:itm:approx} 
        (which is applicable due to Lemma~\ref{FBLDisjoinedness12}) 
        gives $[\spri(T),\spr(T)]\bbT \subseteq \spec(\overline{T})$.
        So assume instead that $\spec(T) \setminus \{0\} \subseteq \bbT$, 
        Since $T$ is not essentially a root of unity, neither is $\overline{T}$. 
        Therefore it follows from Proposition~\ref{prop:hom-root-full-circle}\ref{prop:hom-root-full-circle:itm:proj} 
        that $\spec(\overline{T}) \setminus \{0\} = \bbT$, so $\spec(\overline{T}) = \bbT \cup \{0\}$.

        \item[Case~5:]
        \emph{$0\in\spec(T)$, $T$ is not essentially a root of unity and $0$ is not a pole of the resolvent.}

        If $0$ is an isolated in $\spec(T)$, then it is an essential singularity of the resolvent; 
        hence the claim follows directly from Corollary~\ref{cor:zero-essential-singularity}. 
        
        So suppose now that $0$ is not isolated in $\spec(T)$. 
        We show that $\bbD \subseteq \spec(\overline{T})$; this implies the claim since the spectrum is closed. 
        So assume for a contradiction that there is a number $\lambda \in \bbD \setminus \spec(\overline{T})$. 
        Then $\lambda \not= 0$. 
        There exists a number $\mu_1 \in \partial\spec(T)$ such that $0 < \modulus{\mu_1} < \modulus{\lambda}$. 
        Indeed, first note that $\lambda \not\in \spec(T)$ by Theorem~\ref{thm:spectral-inclusion}. 
        Moreover, since $0$ is a non-isolated point in $\spec(T)$, we can find a $\mu_0 \in \spec(T)$ such that $0 < \modulus{\mu_0} < \modulus{\lambda}$. 
        If we connect $\mu_0$ and $\lambda$ b< a straight line, this line contains a number $\mu_1$ with the claimed properties. 
        On the other hand, there exists a spectral value $\mu_2$ of $T$ with modulus $\modulus{\mu_2} = 1$. 
        Thus, Theorem~\ref{thm:dounut-from-eigenvalues}\ref{thm:dounut-from-eigenvalues:itm:approx} -- which is applicable due to Lemma~\ref{FBLDisjoinedness12} -- implies that $\lambda \in [\modulus{\mu_1}, \modulus{\mu_2}] \bbT \subseteq \spec(\overline{T})$, a contradiction.
        \qedhere
    \end{description}
\end{proof}

Let us consider the assertions of Theorem~\ref{thm:spec-char-general} in a few concrete examples. 

\begin{example}
    If $T \in \bbC^{n \times n}$, the fifth case in Theorem~\ref{thm:spec-char-general} (on the bottom of the right column) cannot occur, but all other cases can. 
    We sketcht the spectra of $T$ and $\overline{T}$ for a few concrete matrices in Figure~\ref{fig:spectra}.
\end{example}

\begin{figure}[t]
    \begin{minipage}[b]{0.23\textwidth}
        \centering
        \begin{adjustbox}{width=\linewidth}
            % case 1
            \begin{tikzpicture} 
                % axis
                \draw [black, ->] (-2,0) to (2,0) node[anchor=west] {$\bbR$}; 
                \draw [black, ->] (0,-1.95) to (0,2) node[anchor=south] {$\iu \bbR$}; 
                % spectrum of T
                \draw [fill] (-1,0) circle [radius=0.05];
                \draw [fill] (-0.5,0.866) circle [radius=0.05];
                % spectrum of T^_
                \draw [fill, white!50!gray] (1,0) circle [radius=0.05];
                \draw [fill, white!50!gray] (0.5,0.866) circle [radius=0.05];
                \draw [fill, white!50!gray] (-0.5,-0.866) circle [radius=0.05];
                \draw [fill, white!50!gray] (0.5,-0.866) circle [radius=0.05];
                % case 1
                \matrix [below right] at (-2,2) { \node {(a)}; \\};
                % caption
                \matrix [below] at (0,-1.95) { \node {$T = \tmatrix{-1&0\\0&\omega}$};\\};
            \end{tikzpicture}
        \end{adjustbox}
    \end{minipage}
    \begin{minipage}[b]{0.23\textwidth}
        \centering
        \begin{adjustbox}{width=\linewidth}
            % case 3 a)
            \begin{tikzpicture}
                % spectrum of T^_
                \draw [white!50!gray,fill] (0,0) circle [radius=1.04];
                \draw [white,fill] (0,0) circle [radius=0.96];
                % axis
                \draw [black, ->] (-2,0) to (2,0) node[anchor=west] {$\bbR$}; 
                \draw [black, ->] (0,-2) to (0,2) node[anchor=south] {$\iu \bbR$}; 
                % spectrum of T
                \draw [fill] (0.54,0.841) circle [radius=0.05];
                % case 3a
                \matrix [below right] at (-2,2) { \node {(c)}; \\ };
                % caption
                \matrix [below] at (0,-2) { \node {$T = \tmatrix{e^\iu}$}; \\};
            \end{tikzpicture}
        \end{adjustbox}
    \end{minipage}
    \begin{minipage}[b]{0.23\textwidth}
        \centering
        \begin{adjustbox}{width=\linewidth}
            % case 3 b)
            \begin{tikzpicture}
                % spectrum of T^_
                \draw [white!50!gray,fill] (0,0) circle [radius=1.04];
                \draw [white,fill] (0,0) circle [radius=0.96];
                % axis
                \draw [black, ->] (-2,0) to (2,0) node[anchor=west] {$\bbR$}; 
                \draw [black, ->] (0,-2) to (0,2) node[anchor=south] {$\iu \bbR$}; 
                % spectrum of T
                \draw [fill] (1,0) circle [radius=0.05];
                % case 3b
                \matrix [below right] at (-2,2) { \node {(e)}; \\ };
                % caption
                \matrix [below] at (0,-2) { \node {$T = \tmatrix{1&1\\0&1}$}; \\};
            \end{tikzpicture}
        \end{adjustbox}
    \end{minipage}
    \begin{minipage}[b]{0.23\textwidth}
        \centering
        \begin{adjustbox}{width=\linewidth}
            % case 3 c)
            \begin{tikzpicture}
                % spectrum of T^_
                \draw [white!50!gray,fill] (0,0) circle [radius=1.5];
                \draw [white,fill] (0,0) circle [radius=0.75];
                % axis
                \draw [black, ->] (-2,0) to (2,0) node[anchor=west] {$\bbR$}; 
                \draw [black, ->] (0,-2) to (0,2) node[anchor=south] {$\iu \bbR$}; 
                % spectrum of T
                \draw [fill] (0.75,0) circle [radius=0.05];
                \draw [fill] (1.5,0) circle [radius=0.05];
                % case 3c
                \matrix [below right] at (-2,2) { \node {(f)}; \\ };
                % caption
                \matrix [below] at (0,-2) { \node {$T = \tmatrix{1&0\\0&2}$}; \\};
            \end{tikzpicture}
        \end{adjustbox}
    \end{minipage}
    
    \begin{minipage}[b]{0.23\textwidth}
        \centering
        \begin{adjustbox}{width=\linewidth}
            % case 2
            \begin{tikzpicture}
                % axis
                \draw [black, ->] (-2,0) to (2,0) node[anchor=west] {$\bbR$}; 
                \draw [black, ->] (0,-1.8) to (0,2) node[anchor=south] {$\iu \bbR$}; 
                % spectrum of T
                \draw [fill] (0,0) circle [radius=0.05];
                \draw [fill] (-1,0) circle [radius=0.05];
                \draw [fill] (-0.5,0.866) circle [radius=0.05];
                % spectrum of T^_
                \draw [fill, white!50!gray] (1,0) circle [radius=0.05];
                \draw [fill, white!50!gray] (0.5,0.866) circle [radius=0.05];
                \draw [fill, white!50!gray] (-0.5,-0.866) circle [radius=0.05];
                \draw [fill, white!50!gray] (0.5,-0.866) circle [radius=0.05];
                % case 2
                \matrix [below right] at (-2,2) { \node {(b)}; \\};
                % caption
                \matrix [below] at (0,-1.8) { \node {$T = \tmatrix{0&0&0&\\0&-1&0\\0&0&\omega}$};\\};
            \end{tikzpicture}
        \end{adjustbox}
    \end{minipage}
    \begin{minipage}[b]{0.23\textwidth}
        \centering
        \begin{adjustbox}{width=\linewidth}
            % case 4 a)
            \begin{tikzpicture}
                % spectrum of T^_
                \draw [white!50!gray,fill] (0,0) circle [radius=1.04];
                \draw [white,fill] (0,0) circle [radius=0.96];
                % axis
                \draw [black, ->] (-2,0) to (2,0) node[anchor=west] {$\bbR$}; 
                \draw [black, ->] (0,-2) to (0,2) node[anchor=south] {$\iu \bbR$}; 
                % spectrum of T
                \draw [fill] (0,0) circle [radius=0.05];
                \draw [fill] (0.54,0.841) circle [radius=0.05];
                % case 4a
                \matrix [below right] at (-2,2) { \node {(d)}; \\ };
                % caption
                \matrix [below] at (0,-2) { \node {$T = \tmatrix{0&0\\0&e^\iu}$}; \\};
            \end{tikzpicture}
        \end{adjustbox}
    \end{minipage}
    \begin{minipage}[b]{0.23\textwidth}
        \centering
        \begin{adjustbox}{width=\linewidth}
            % case 4 b)
            \begin{tikzpicture}
                % spectrum of T^_
                \draw [white!50!gray,fill] (0,0) circle [radius=1.04];
                \draw [white,fill] (0,0) circle [radius=0.96];
                % axis
                \draw [black, ->] (-2,0) to (2,0) node[anchor=west] {$\bbR$}; 
                \draw [black, ->] (0,-1.8) to (0,2) node[anchor=south] {$\iu \bbR$}; 
                % spectrum of T
                \draw [fill] (0,0) circle [radius=0.05];
                \draw [fill] (1,0) circle [radius=0.05];
                % case 4b
                \matrix [below right] at (-2,2) { \node {(f)}; \\ };
                % caption
                \matrix [below] at (0,-1.8) { \node {$T = \tmatrix{0&0&0\\0&1&1\\0&0&1}$}; \\};
            \end{tikzpicture}
        \end{adjustbox}
    \end{minipage}
    \begin{minipage}[b]{0.23\textwidth}
        \centering
        \begin{adjustbox}{width=\linewidth}
            % case 4 c)
            \begin{tikzpicture}
                % spectrum of T^_
                \draw [white!50!gray,fill] (0,0) circle [radius=1.5];
                \draw [white,fill] (0,0) circle [radius=0.75];
                % axis
                \draw [black, ->] (-2,0) to (2,0) node[anchor=west] {$\bbR$}; 
                \draw [black, ->] (0,-1.8) to (0,2) node[anchor=south] {$\iu \bbR$}; 
                % spectrum of T
                \draw [fill] (0,0) circle [radius=0.05];
                \draw [fill] (0.75,0) circle [radius=0.05];
                \draw [fill] (1.5,0) circle [radius=0.05];
                % case 4c
                \matrix [below right] at (-2,2) { \node {(h)}; \\ };
                % caption
                \matrix [below] at (0,-1.8) { \node {$T = \tmatrix{0&0&0\\0&1&0\\0&0&2}$}; \\};
            \end{tikzpicture}
        \end{adjustbox}
    \end{minipage}
    \caption{\label{fig:spectra}
        In each figure $\spec(T)$ is depicted as black dots and $\spec(\overline{T})$ is depicted in gray. 
        In~(a) and~(b) we use the number $\omega \coloneqq e^{\frac{2\pi\iu}{3}}$.
    }
\end{figure}
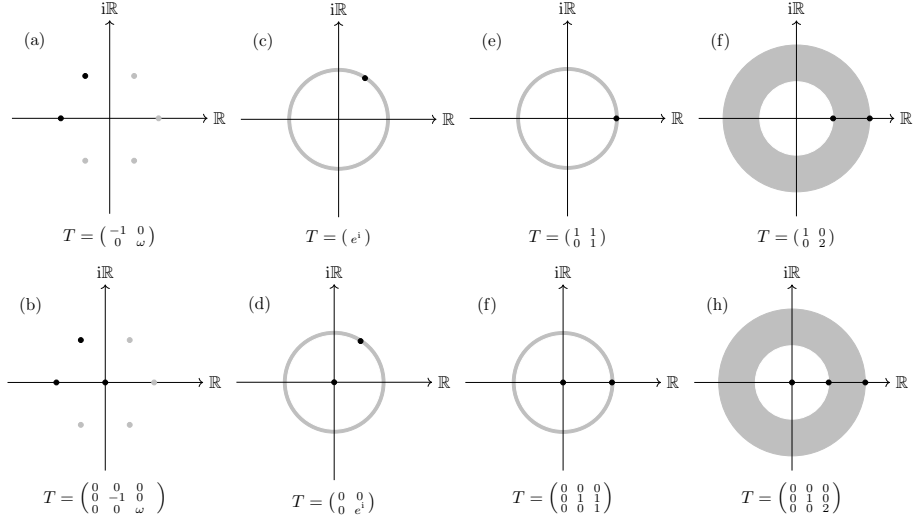

Finally, let us also consider several infinite-dimensional examples.

\begin{examples}
    \label{exas:inf-dim-pictures}
    \leavevmode
    \begin{enumerate}[label=(\alph*)]
        \item\label{exas:inf-dim-pictures:itm:volterra}
        Consider the Volterra operator $V$ on $L^2([0,1])$ that is given by $(Vf)(y) = \int_0^y f(x) \dx x$. 
        It has spectrum $\spec(V) = \{0\}$ and $0$ is an essential singularity of the resolvent. 
        Therefore the block diagonal operator $T:L^2([0,1])\times\bbC\to L^2([0,1])\times\bbC$ 
        with block entries $V$ and $\id_\bbC$ has spectrum $\spec(T) = \{0,1\}$, 
        and again $0$ is an essential singularity of its resolvent. 
        The spectra of $T$ and $\overline{T}$ are shown in Figure~\ref{fig:essential-singularity}(a).

        \item\label{exas:inf-dim-pictures:itm:multiplication}
        Let $p \in [1,\infty]$ and let $T \colon \ell^p \to \ell^p$ be given by 
        $Tx = (\tfrac{1}{n}x_n)$ for each $x = (x_n) \in \ell^p$.
        Its spectrum is $\spec(T) = \pntSpec(T) = \{\tfrac{1}{n}\mid n\in\bbN\}$
        Both $\spec(T)$ and $\spec(\overline{T})$ are shown in Figure~\ref{fig:essential-singularity}(b).
        
        \item\label{exas:inf-dim-pictures:itm:shift-modified}
        Let $p \in [1,\infty]$ and let $T:\ell^p\to\ell^p$ be defined by $Tx = \tfrac12 (2x_1,0,x_2,x_3,x_4,\dots)$. 
        It spectrum is $\spec(T) = \frac{1}{2}\overline{\bbD} \cup \{1\}$. 
        Both $\spec(T)$ and $\spec(\overline{T})$ are shown in Figure~\ref{fig:essential-singularity}(c).
    \end{enumerate}
\end{examples}

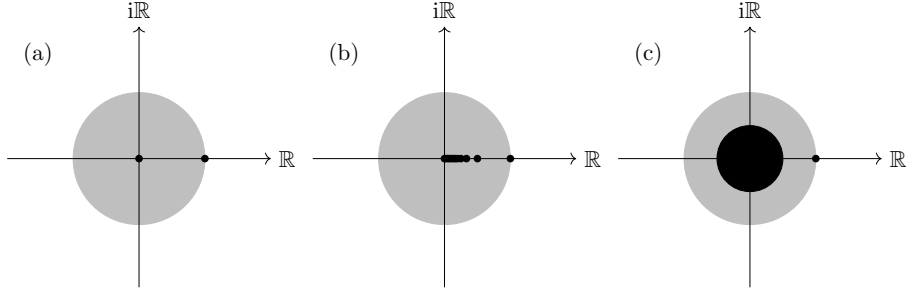
\begin{figure}[t]
        \begin{minipage}[b]{0.31\textwidth}
            \centering
            \begin{adjustbox}{width=\linewidth}
                % case 1
                \begin{tikzpicture} 
                    % spectrum of T^_
                    \draw [fill, white!50!gray] (0,0) circle [radius=1];
                    % axis
                    \draw [black, ->] (-2,0) to (2,0) node[anchor=west] {$\bbR$}; 
                    \draw [black, ->] (0,-1.95) to (0,2) node[anchor=south] {$\iu \bbR$}; 
                    % spectrum of T
                    \draw [fill] (0,0) circle [radius=0.05];
                    \draw [fill] (1,0) circle [radius=0.05];
                    % case 1
                    \matrix [below right] at (-2,2) { \node {(a)}; \\};
                    % caption
                    %\matrix [below] at (0,-1.95) { \node {$T$};\\};
                \end{tikzpicture}
            \end{adjustbox}
        \end{minipage}
        % case 2
        \begin{minipage}[b]{0.31\textwidth}
            \centering
            \begin{adjustbox}{width=\linewidth}
            \begin{tikzpicture} 
                % spectrum of T^_
                \draw [fill, white!50!gray] (0,0) circle [radius=1];
                % axis
                \draw [black, ->] (-2,0) to (2,0) node[anchor=west] {$\bbR$}; 
                \draw [black, ->] (0,-1.95) to (0,2) node[anchor=south] {$\iu \bbR$}; 
                % spectrum of T
                \draw [fill] (0,0) circle [radius=0.05];
                \draw [fill] (1,0) circle [radius=0.05];
                \draw [fill] (1/2,0) circle [radius=0.05];
                \draw [fill] (1/3,0) circle [radius=0.05];
                \draw [fill] (1/4,0) circle [radius=0.05];
                \draw [fill] (1/5,0) circle [radius=0.05];
                \draw [fill] (1/6,0) circle [radius=0.05];
                \draw [fill] (1/7,0) circle [radius=0.05];
                \draw [line width = 1.15mm, black] (0,0) to (1/8,0);
                % case 1
                \matrix [below right] at (-2,2) { \node {(b)}; \\};
                % caption
                %\matrix [below] at (0,-1.95) { \node {$S$};\\};
            \end{tikzpicture}
            \end{adjustbox}
        \end{minipage}
        % case 3
        \begin{minipage}[b]{0.31\textwidth}
            \centering
            \begin{adjustbox}{width=\linewidth}
                \begin{tikzpicture}
                    % spectrum of T^_
                    \draw [fill, white!50!gray] (0,0) circle [radius=1];
                    % axis
                    \draw [black, ->] (-2,0) to (2,0) node[anchor=west] {$\bbR$}; 
                    \draw [black, ->] (0,-1.95) to (0,2) node[anchor=south] {$\iu \bbR$}; 
                    % spectrum of T
                    \draw [fill, black] (0,0) circle [radius=1/2];
                    \draw [fill] (1,0) circle [radius=0.05];
                    % case 1
                    \matrix [below right] at (-2,2) { \node {(c)}; \\};
                    % caption
                    %\matrix [below] at (0,-1.95) { \node {$R$};\\};
                \end{tikzpicture}
            \end{adjustbox}
        \end{minipage}
    \caption{\label{fig:essential-singularity}
        Spectra of the operators from Examples~\ref{exas:inf-dim-pictures}.
        In each figure $\spec(T)$ is depicted as black dots and $\spec(\overline{T})$ is depicted in gray.
    }
\end{figure}

\subsection*{Acknowledgements} 
Phillip Krokor was funded by the Deutsche Forschungsgemeinschaft (DFG, German Research Foundation) -- 541471251.

\bibliographystyle{plainurl}
\bibliography{literature}

\end{document}